\documentclass[11pt]{amsart}
\usepackage{amsmath}
\usepackage{mathtools}
\usepackage{amsfonts}
\usepackage{amssymb}
\usepackage{graphicx}
\usepackage{tikz-cd}
\usepackage[all,cmtip]{xy} 
\usepackage{fullpage}

\usepackage{aliascnt} 

\usepackage[T1]{fontenc}

\newcommand{\SL}{\mathrm{SL}}

\newcommand{\Q}{\mathbb{Q}}

\newcommand{\C}{\mathbb{C}}

\newcommand{\GL}{\mathrm{GL}}

\usepackage{verbatim}
\usepackage{graphicx}
\usepackage{nomencl}
\DeclareFontEncoding{OT2}{}{} 
\newcommand{\Z}{\mathbb{Z}}

\theoremstyle{plain}
\newtheorem{theorem}{Theorem}[section]
\newaliascnt{corollary}{theorem}
\newtheorem{corollary}[corollary]{Corollary}
\aliascntresetthe{corollary}

\newaliascnt{corollary2}{theorem}

\aliascntresetthe{corollary2}

\newaliascnt{lemma}{theorem}
\newtheorem{lemma}[lemma]{Lemma}
\aliascntresetthe{lemma}
\newaliascnt{proposition}{theorem}
\newtheorem{proposition}[proposition]{Proposition}
\aliascntresetthe{proposition}

\newaliascnt{hypotheses}{theorem}

\aliascntresetthe{hypotheses}

\theoremstyle{definition}
\newaliascnt{definition}{theorem}
\newtheorem{definition}[definition]{Definition}
\aliascntresetthe{definition}

\newaliascnt{example}{theorem}

\aliascntresetthe{example}

\newaliascnt{remark}{theorem}

\aliascntresetthe{remark}

\newaliascnt{remarks}{theorem}

\aliascntresetthe{remarks}

\numberwithin{equation}{section}

\usepackage{calrsfs}

\usepackage{}
\usepackage[pdftex,breaklinks,colorlinks,
citecolor=blue,
linkcolor=blue,
urlcolor=blue]{hyperref}

\everymath{\displaystyle}

\begin{document}
\title{Eisenstein cohomology classes for $\mathrm{GL}_N$ over imaginary quadratic fields}

\author{Nicolas Bergeron}
\address{DMA UMR 8553 ENS / PSL and Sorbonne Université, 75005, Paris, France}
\email{nicolas.bergeron@ens.fr}
\urladdr{https://sites.google.com/view/nicolasbergeron/accueil}

\author{Pierre Charollois}
\address{Sorbonne Universit\'e, IMJ--PRG, CNRS, Univ Paris Diderot, F-75005, Paris, France}
\email{pierre.charollois@imj-prg.fr}
\urladdr{http://people.math.jussieu.fr/~charollois}

\author{Luis E. Garcia}
\address{Department of Mathematics, University College London, Gower Street, London WC1E 6BT, United Kingdom}
\email{l.e.garcia@ucl.ac.uk}

\begin{abstract}
We study the arithmetic of degree $N-1$ Eisenstein cohomology classes for the locally symmetric spaces attached to $\mathrm{GL}_N$ over an imaginary quadratic field $k$. Under natural conditions we evaluate these classes on $(N-1)$-cycles associated to degree $N$ extensions $L/k$ as linear combinations of generalised Dedekind sums. As a consequence we prove a remarkable conjecture of Sczech and Colmez expressing critical values of $L$-functions attached to Hecke characters of $F$ as polynomials in Kronecker--Eisenstein  series evaluated at torsion points on elliptic curves with multiplication by $k$. We recover in particular the algebraicity of these critical values. 
\end{abstract}
\maketitle

\setcounter{tocdepth}{1}
\tableofcontents

\section{Introduction}

The relationship between Eisenstein series, the cohomology of arithmetic groups and special values of $L$-functions has been studied extensively.

A classical example is that of weight $2$ Eisenstein series associated to a pair $(\alpha , \beta) \in (\Q / \Z)^2$. Such a series can be defined as limits of finite sums:
$$E_{2,(\alpha , \beta)} (\tau) = \lim_{M \to +\infty} \sum_{m=-M}^M \left(  \lim_{N \to +\infty} \sideset{}{'} \sum_{n=-N}^N \frac{e^{2i\pi (m \alpha + n \beta)}}{(m\tau +n)^2} \right) \quad (\tau \in \mathcal{H}).$$
Here the prime on the sum means that we exclude the term $(m,n)=(0,0)$. When $(\alpha , \beta) \neq (0,0)$, the holomorphic $1$-form 
$E_{2,(\alpha , \beta)} (\tau ) d\tau$ on Poincar\'e's upper half-plane $\mathcal{H}$ 
is invariant under any subgroup $\Gamma \subset \mathrm{SL}_2 (\Z)$ that fixes $(\alpha , \beta)$ modulo $\Z^2$. This holomorphic form then represents a cohomology class in $H^1 (\Gamma , \C)$. A remarkable feature of these classes is that they are rational and even almost integral. 

A convenient and compact way to state the precise integrality properties of these cohomology classes is to consider for each prime integer the `$p$-smoothed Eisenstein series'
\begin{equation*}
\begin{split}
E_{2 , (\alpha , \beta)}^{(p)} (\tau ) & = \sum_{j=1}^{p} E_{2 , (\alpha +j/p , \beta)} (\tau ) - p E_{2 , (\alpha , \beta)} (\tau) \\
& = p (E_{2 , (p \alpha , \beta)} (p \tau) - E_{2 , (\alpha , \beta)} (\tau) ),
\end{split}
\end{equation*}
and suppose furthermore that $\Gamma \subset \Gamma_0 (p)$. Then $E_{2, (\alpha , \beta)}^{(p)} (\tau )$ yields a homomorphism $\Phi_{(\alpha , \beta)}^{(p)} : \Gamma  \to \C$ 
by the rule 
$$\Phi_{(\alpha , \beta)}^{(p)} (\gamma)  = \int_{\tau_0}^{\gamma \tau_0} E_{2, (\alpha , \beta)}^{(p)} (\tau ) d\tau,$$
for any base point $\tau_0 \in \mathcal{H}$, and it is classical (see e.g. \cite[Theorem 13]{Siegel}) that we have:
\begin{equation} \label{E1}
\Phi_{(\alpha , \beta)}^{(p)} \left( \begin{array}{cc} a & b \\ c & d \end{array} \right) = \left\{ \begin{array}{ll}
0, & \mbox{if } c=0, \\
(2\pi)^2 \cdot \mathrm{sign} (c) \cdot \left( D_{p\alpha , \beta} \left( \frac{pa}{|c|} \right) -  p D_{\alpha , \beta} \left( \frac{a}{|c|} \right) \right), & \mbox{otherwise}.
\end{array} \right.
\end{equation}
Here $D_{\alpha , \beta}$ denotes the {\it generalised Dedekind sum}
$$D_{\alpha , \beta} \left( \frac{a}{c} \right) = \sum_{j=1}^{c} \left( \left(  \frac{j-\beta}{c} \right) \right) \ \left( \left(  \frac{a (j-\beta)}{c} - \alpha \right) \right) \quad \mbox{for} \quad c>0 \quad \mbox{and} \quad (a,c) =1,$$
where the symbol $((x))$ is defined by 
$$((x)) = \left\{ \begin{array}{ll} 
x- [x] - 1/2 & \mbox{if } x \mbox{ is not an integer}, \\
0 & \mbox{if } x \mbox{ is an integer}.
\end{array} \right.$$
These sums define rational numbers and enjoy many beautiful arithmetical properties, see e.g. \cite{Rademacher}. On the other hand a formula of Siegel \cite{Siegel} expresses the values at nonpositive integers of the $\zeta$-functions attached to real quadratic fields as periods of Eisenstein series. The expression \eqref{E1} can therefore be turned into a very explicit expression for these special values. This implies in particular that they are essentially integral which is the key input in the construction by Coates and Sinnott \cite{CoatesSinnott} of $p$-adic $L$-functions over real quadratic fields. 

Using Selberg's and Langlands' theory of Eisenstein series Harder has vastly generalised the above mentioned `Eisenstein cohomology classes.'  In \cite{Harder} he constructed a complement to the cuspidal cohomology for the group $\mathrm{GL}_2$ over number fields and manage to construct rational representatives. In a more recent work Harder even managed to address some integrality properties of these classes, see \cite{Harder2}. However for the group $\GL_N$ it is hard to check that Eisenstein classes are rational and the automorphic form theory is not yet adapted to the study of integrality properties of these classes. 

For $\GL_N$ over the field of rational numbers, Nori \cite{Nori} and Sczech \cite{Sczech} have proposed constructions of Eisenstein cohomology classes that have turned out to be very efficient in practice to study the fine arithmetical properties of $L$-functions over totally real number fields, see e.g. \cite{GS,CharolloisDasgupta,CDG,BKL}. Sczech's approach more generally gives formulas analogous to \eqref{E1}. The goal of this paper is to prove similar formulas for the group $\GL_N$ over an imaginary quadratic field $k$. As a consequence we prove a remarkable conjecture of Sczech and Colmez \cite[Conjecture p. 205]{Colmez} expressing critical values of $L$-functions attached to Hecke characters of finite extension of $k$ as polynomials in Kronecker--Eisenstein series evaluated at torsion points on elliptic curves with multiplication by $k$. 

We now describe in more details our main results. 

\subsection{An Eisenstein cocycle for imaginary quadratic fields}

Fix a positive integer $N \geq 2$. Let $k$ be a quadratic imaginary field with ring of integers $\mathcal{O}$ and let $\mathfrak{p} \subset \mathcal{O}$ be an ideal of prime norm $\mathrm{N}\mathfrak{p}$. Our first main result is the construction of an $(N-1)$-cocycle for the level $\mathfrak{p}$ congruence subgroup
\begin{equation}
\Gamma_0(\mathfrak{p})= \left\{ \left. \begin{pmatrix} a & {^t}b \\ c & D \end{pmatrix} \in \mathrm{SL}_N(\mathcal{O}) \ \right| \  a \in \mathcal{O}, b \in \mathcal{O}^{N-1}, c \in \mathfrak{p}^{N-1}, D \in M_{N-1}(\mathcal{O}) \right\}
\end{equation} 
of $\mathrm{SL}_N(\mathcal{O})$ taking values in the space of polynomials in certain classical series called Kronecker--Eisenstein series.

Let us recall the definition of Kronecker--Eisenstein series (see \cite[VIII.\S  12]{WeilEllipticFunctions}). We fix once and for all an embedding $\sigma:k \rightarrow \mathbb{C}$. For a fractional ideal $\mathfrak{I}$ of $k$ and non-negative integers $p$ and $q$, define
\begin{equation}
K^{p,q}(z,\mathfrak{I},s) = p! \sum_{\lambda \in \sigma(\mathfrak{I})} \frac{\overline{z+\lambda}^q}{(z+\lambda)^{p+1}|z+\lambda|^{s}}
\end{equation}
where we assume that $z \in \mathbb{C}$ satisfies $z \notin \sigma(\mathfrak{I})$. The series converges when $\mathrm{Re}(s)>1+q-p$ and has analytic continuation to $s \in \mathbb{C}$ that is regular at $s=0$. It is a classical result due to Damerell \cite{Damerell} that the values at $s=0$ have the following algebraicity property:
\begin{equation} \label{E:Damerell}
K^{p,q}(z_0,\mathfrak{I},0) \in \Omega_\infty^{1+p+q} \pi^{-q} \overline{\mathbb{Q}} \quad \text{for } z_0 \in k \backslash \sigma(\mathfrak{I}).
\end{equation}
Here $\Omega_\infty$ denotes any period of a $\overline{\mathbb{Q}}$-rational holomorphic differential against a non-zero rational homology class on an elliptic curve with CM by $k$ defined over $\overline{\mathbb{Q}}$. In fact these series have almost integral values; we refer to \cite{Katz} for precise results. 

Next we introduce polynomials in the series $K^{p,q}(z,\mathfrak{I},s)$. For fractional ideals $\mathfrak{I}_1,\ldots,\mathfrak{I}_N$ and multi-indices $I=(i_1,\ldots,i_N) \in \mathbb{Z}_{\geq 0}^N$ and $J=(j_1,\ldots,j_N) \in \mathbb{Z}_{\geq 0}^N$, we set
\begin{equation}
K^{I,J}(z, \mathfrak{I}_1 \oplus \cdots \oplus \mathfrak{I}_N, s) = K^{i_1,j_1}(z_1,\mathfrak{I}_1,s) \cdots K^{i_N,j_N}(z_N,\mathfrak{I}_N,s).
\end{equation}
More generally, for an $\mathcal{O}$-lattice $\Lambda \subset k^N$, we pick fractional ideals $\mathfrak{I}_1,\ldots,\mathfrak{I}_N$ such that $\mathfrak{I}_1 \oplus \cdots \oplus \mathfrak{I}_N$ has finite index in $\Lambda$ and set
\begin{equation}
\begin{split}
K^{I,J}(z, \Lambda, s) &= \sum_{\lambda \in \Lambda/ \mathfrak{I}_1 \oplus \cdots \oplus \mathfrak{I}_N} K^{I,J}(z+\sigma(\lambda), \mathfrak{I}_1 \oplus \cdots \oplus \mathfrak{I}_N, s) \\
&= \sum_{\lambda \in \Lambda} \prod_{1 \leq k \leq N} i_k !  \frac{\overline{z_k+\sigma(\lambda_k)}^{j_k}}{(z_k+\sigma(\lambda_k))^{i_k+1}|z_k+\sigma(\lambda_k)|^{s}}.
\end{split}
\end{equation}
As the last expression shows, $K^{I,J}(z,\Lambda,s)$ does not depend on the choice of the fractional ideals $\mathfrak{I}_k$. We set
\begin{equation}
K^{I,J}(z,\Lambda):= K^{I,J}(z,\Lambda,0).
\end{equation}
Each $K^{I,J}(z,\Lambda)$ defines a smooth function on an open subset of $\mathbb{C}^N$ obtained by removing all $\Lambda$-translates of a finite number of hyperplanes. We write
\begin{equation}
\mathcal{F}=\langle K^{I,J}(\gamma z, \Lambda) \ | \ \gamma \in \mathrm{SL}_N(k), \Lambda \subset k^N \text{ an } \mathcal{O}-\text{lattice} \rangle
\end{equation}
for the $\mathbb{C}$-span of $\mathrm{SL}_N(k)$-translates of all functions $K^{I,J}(z,\Lambda)$.

Next we introduce the $\mathfrak{p}$-smoothed series
\begin{equation}
K^{I,J}_\mathfrak{p}(z,\mathcal{O}^N) = K^{I,J}(z,\mathfrak{p}^{-1} \oplus \mathcal{O}^{N-1}) - \mathrm{N}\mathfrak{p} \cdot K^{I,J}(z,\mathcal{O}^N)
\end{equation}
and, for $A \in M_N(\mathcal{O})$, we define the generalised Dedekind sum $D_\mathfrak{p}^{I,J}(z,A)$ by
\begin{equation}
D_\mathfrak{p}^{I,J}(z,A) = \det A^{-1}  K_\mathfrak{p}^{I,J}(A^{-1}z,A^{-1}\mathcal{O}^N)
\end{equation}
if $A$ is invertible and set $D_\mathfrak{p}^{I,J}(z,A)=0$ otherwise. These sums are natural generalisations of Dedekind sums for imaginary quadratic fields and $N$ variables. 

Our first theorem shows that the series $D^{I,J}_\mathfrak{p}(z,A)$ can be combined into a homogeneous $(N-1)$-cocycle for $\Gamma_0(\mathfrak{p})$. In the following statement, for a multi-index $I \in \mathbb{Z}_{\geq 0}^N$, we write $|I|=i_1+\cdots + i_N$. When $I$ (resp. $J$) runs over multi-indices with $|I|=p$ (resp.  $|J|=q$),  the vectors 
$$
e^I:=e_1^{i_1}\cdots e_N^{i_n} \in \mathrm{Sym}^p \mathbb{C}^N \quad (\text{resp. } \overline{e}^{J}:=\overline{e_1}^{j_1} \cdots \overline{e_N}^{j_N}  \in \overline{\mathrm{Sym}^q \mathbb{C}^N})
$$
form a basis of $\mathrm{Sym}^p \mathbb{C}^N$ (resp. of $\overline{\mathrm{Sym}^q \mathbb{C}^N})$.

\begin{theorem} \label{Theorem1Intro}
Given $\gamma_1,\ldots,\gamma_N \in \Gamma_0(\mathfrak{p})$, define
$$
A(\underline{\gamma})=(\gamma_1 e_1|\cdots|\gamma_N e_1) \in M_N(\mathcal{O}).
$$
The map
$$
\mathbf{\Phi}^{p,q}_\mathfrak{p}: \Gamma_0(\mathfrak{p})^N \to \mathcal{F} \otimes \mathrm{Sym}^p{\mathbb{C}^N} \otimes \overline{\mathrm{Sym}^q\mathbb{C}^N}
$$
given by
$$
\mathbf{\Phi}^{p,q}_\mathfrak{p}(z,\underline{\gamma}) = \sum_{|I|=p, |J|=q} D_{\mathfrak{p}}^{I,J}(z,A(\underline{\gamma})) \otimes A(\underline{\gamma})(e^{I} \otimes \overline{e}^{J})
$$
is a homogeneous $(N-1)$-cocycle. Here the sum runs over all multi-indices $I,J \in \mathbb{Z}_{\geq 0}^N$ with $|I|=p$ and $|J|=q$.
\end{theorem} 

More concretely, the cocycle property of $\mathbf{\Phi}^{p,q}_\mathfrak{p}$ means that
\begin{equation}
\begin{split}
\mathbf{\Phi}^{p,q}_\mathfrak{p} (\gamma z,\gamma \gamma_1,\ldots, \gamma \gamma_N) &= \gamma \mathbf{\Phi}^{p,q}_\mathfrak{p} (\gamma  z,\gamma_1,\ldots,\gamma_N) \\
&= \sum_{|I|=p, \ |J|=q} D_\mathfrak{p}^{I,J}(z,A(\underline{\gamma})) \otimes \gamma A(\underline{\gamma})(e^I \otimes \overline{e}^J)
\end{split}
\end{equation}
for any $\gamma,\gamma_1,\ldots,\gamma_N \in \Gamma_0(\mathfrak{p})$, and
$$
\sum_{1 \leq k \leq N+1} (-1)^{k-1} \mathbf{\Phi}^{p,q}_\mathfrak{p} (z,\gamma_1,\ldots,\widehat{\gamma_k},\ldots,\gamma_{N+1})=0
$$
for any $\gamma_1,\ldots,\gamma_{N+1} \in \Gamma_0(\mathfrak{p})$ (here as usual the notation $\widehat{\gamma_k}$ means that the term $\gamma_k$ is to be omitted).

More generally, in the body of the paper we introduce a cocycle $\mathbf{\Phi}^{p,q}_\mathfrak{p} (z,\underline{\gamma},\Lambda(\mathfrak{I}))$ for the $\mathcal{O}$-lattice $\Lambda(\mathfrak{I})=\mathfrak{I}^{-1} \oplus \mathcal{O}^{N-1}$ of $k^N$.

\subsection{Application to critical values of Hecke $L$-functions} \label{S:intro12}
We refer to \cite[Section 1]{Fargues} or \cite{ClozelMotifs,SerreAbelianl-adic,RohrlichRootNumbers} for generalities on Hecke characters.
Let $L/k$ be a field extension of degree $N > 1$ and let $n:L \to k$ denote the norm map. We fix an algebraic Hecke character $\psi_k$ of $k$ of infinity type $(p,q) \in \mathbb{Z}^2$ and a Dirichlet character $\chi$ of $L$, 
and consider the algebraic Hecke character
\begin{equation}
\label{eq:alg_Hecke_char_intro}
\phi = \chi \cdot ( \psi_k \circ n )
\end{equation}
of $L$. We denote the conductor of $\phi$ by $\mathfrak{f}$, so that for $\alpha \equiv 1 \mod \mathfrak{f}$ we have
$$
\phi((\alpha)) = n(\alpha)^{p} \overline{n(\alpha)}^q.
$$
Note that if $k$ is a maximal CM field in $L$ then any algebraic Hecke character $\phi$ of $L$ is of the above form. 

The Hecke $L$-function of $\phi$ is
$$
L(\phi,s) =  \prod_{(\mathfrak{P} , \mathfrak{f})=1} (1-\phi (\mathfrak{P}) \mathrm{N}\mathfrak{P}^{-s})^{-1} = \sum_{(\mathfrak{a},\mathfrak{f})=1} \phi(\mathfrak{a})\mathrm{N}\mathfrak{a}^{-s},
$$
where the sum, resp. the product, runs over integral ideals $\mathfrak{a}$, resp. prime ideals $\mathfrak{P}$, of $\mathcal{O}_L$ coprime to $\mathfrak{f}$. The global L-function of $\phi$ is $\Lambda(\phi,s) =  L_{\infty}(\phi,s)L(\phi,s)$, where
\[
L_\infty(\phi,s) = \prod_{v|\infty} \Gamma(\phi_v,s).
\]
Here each $\phi_v$ with $v | \infty$ is of the form 
$$z^{p} \overline{z}^{q} = (z\overline{z})^{w/2} \left( \frac{z}{\overline{z}} \right)^{(p-q)/2},$$
with $w=p+q$ (the weight), and 
$$\Gamma (\phi_v ,s) = 2 (2\pi)^{-(s-w/2+|p-q|/2)} \Gamma \left( s-\frac{w}{2}+\frac{|p-q|}{2} \right).$$
The value $L(\phi,s_0)$ at an integer $s_0 \in \mathbb{Z}$ is said to be critical if and only if  
$$\mathrm{ord}_{s=s_0} L_\infty(\phi,s) =\mathrm{ord}_{s=s_0} L_\infty (\phi^{-1},1-s)= 0.$$
In our case this is equivalent to 
$$
\frac{w}{2}-\frac{|p-q|}{2} < s_0 <1+\frac{w}{2}+\frac{|p-q|}{2}.
$$

Our second main result is that for critical $s_0$ the value $L(\phi,s_0)$ can be expressed as an explicit polynomial in Kronecker--Eisenstein series; answering positively a conjecture of Sczech and Colmez \cite[Conjecture p. 205]{Colmez}. Note that the complex conjugate $\overline{\psi_k}$ has weight $(q,p)$ and that multiplying $\phi$ by an integral power of the norm character shifts $s$ by an integer. Thus we may assume that $p<0$ and $q \geq 0$ and consider only the critical value $L(\phi,0)$.

Our result is more conveniently expressed in terms of partial zeta functions, as follows. For integers $p,q$ and integral ideals $\mathfrak{a}$, $\mathfrak{f}$ of $\mathcal{O}_L$, define
\begin{equation}
\zeta_\mathfrak{f}^{p,q}(\mathfrak{a},s)= \sideset{}{'}\sum_{x \in U(\mathfrak{f})\backslash 1+\mathfrak{fa}^{-1}} \frac{\overline{n(x)}^q}{n(x)^{p+1} |n(x)|^{2s}}, \qquad \mathrm{Re}(s) \gg 0.
\end{equation}
Here $U(\mathfrak{f})$ denotes the group of units of $\mathcal{O}_L^\times$ that are congruent to $1$ modulo $\mathfrak{f}$.
(Since $u\overline{u}=1$ for every $u \in \mathcal{O}^\times$, this is well-defined provided that $p+q+1$ is divisible by the order of the subgroup $n(U(\mathfrak{f}))$ of $\mathcal{O}^\times$, which we assume.) Choosing integral ideals $\mathfrak{a}_1,\ldots,\mathfrak{a}_r$ giving a system of representatives for the ray class group $C_\mathfrak{f}$ we can write
$$
L(\phi,s) = \sum_{j} \phi(\mathfrak{a}_j)\mathrm{N}\mathfrak{a}_j^{-s}\zeta^{-p-1,q}_{\mathfrak{f}}(\mathfrak{a}_j,s).
$$

Given two distinct prime ideals $\mathfrak{P}$ and $\tilde{\mathfrak{P}}$ of $\mathcal{O}_L$ coprime to $\mathfrak{f}$ and $\mathfrak{a}$, we define also the `smoothed' partial zeta functions
\begin{equation}
\begin{split}
\zeta^{p,q}_{\mathfrak{f},\mathfrak{P}}(\mathfrak{a},s)&= \mathrm{N}\mathfrak{P}^{-s} \zeta^{p,q}_\mathfrak{f}(\mathfrak{aP},s)-\mathrm{N}\mathfrak{P}^{1-s} \zeta^{p,q}_\mathfrak{f}(\mathfrak{a},s) \\
\zeta^{p,q}_{\mathfrak{f},\mathfrak{P},\tilde{\mathfrak{P}}}(\mathfrak{a},s)&=\mathrm{N}\tilde{\mathfrak{P}}^{-s} \zeta^{p,q}_\mathfrak{f,\mathfrak{P}}(\mathfrak{a\tilde{P}},s)-\mathrm{N}\tilde{\mathfrak{P}}^{-s} \zeta^{p,q}_{\mathfrak{f},\mathfrak{P}}(\mathfrak{a},s).
\end{split}
\end{equation}
These modified zeta functions appear in an expression for $L(\phi,s)$ with modified Euler factors at $\mathfrak{P}$ and $\tilde{\mathfrak{P}}$. Namely, setting 
$$L_\mathfrak{P} (\phi \cdot \mathrm{N},s)=(1-\phi(\mathfrak{P})\mathrm{N}\mathfrak{P}^{1-s})^{-1}, \quad  
L_{\tilde{\mathfrak{P}}}(\phi,s)=(1-\phi(\tilde{\mathfrak{P}})\mathrm{N}\tilde{\mathfrak{P}}^{-s})^{-1},$$ 
and using the fact that $\mathfrak{a}_1 \mathfrak{P},\ldots,\mathfrak{a}_r \mathfrak{P}$ is also a system of representatives of $C_\mathfrak{f}$, we have
\begin{equation}
\begin{split}
L_{\mathfrak{P}}(\phi \cdot \mathrm{N},s)^{-1}L(\phi,s) &= \sum_{j} \phi(\mathfrak{a}_j \mathfrak{P})\mathrm{N}\mathfrak{a}_j ^{-s}\zeta^{-p-1,q}_{\mathfrak{f},\mathfrak{P}}(\mathfrak{a}_j,s), \\
L_{\tilde{\mathfrak{P}}}(\phi,s)^{-1}L_{\mathfrak{P}}(\phi \cdot \mathrm{N},s)^{-1}L(\phi,s) &=\sum_{j} \phi(\mathfrak{a}_j \mathfrak{P} \tilde{\mathfrak{P}})\mathrm{N}\mathfrak{a}_j^{-s}\zeta^{-p-1,q}_{\mathfrak{f},\mathfrak{P},\tilde{\mathfrak{P}}}(\mathfrak{a}_j,s),
\end{split}
\end{equation}

\autoref{Theorem2Intro} below shows that, for appropriate choices of $\mathfrak{P}$ and $\tilde{\mathfrak{P}}$, the zeta function $\zeta_{\mathfrak{f},\mathfrak{P},\tilde{\mathfrak{P}}}^{-p-1,q}(\mathfrak{a}_j,s)$ can be expressed using the Eisenstein cocycle of \autoref{Theorem1Intro}.

Let $U(\mathfrak{f})^1 \subseteq U(\mathfrak{f})$ be the subgroup of units of relative norm one, and let $U(\mathfrak{f})' \subseteq U(\mathfrak{f})^1$ be a torsion--free subgroup that maps bijectively to $U(\mathfrak{f})^1/U(\mathfrak{f})^1_{\mathrm{tors}}$. We also fix a isomorphism $\alpha: L \to k^N$ of $k$-vector spaces and denote the $\mathcal{O}$-lattice $\alpha(\mathfrak{fa}^{-1}) \subset k^N$ by $\Lambda(\mathfrak{fa}^{-1})$.\footnote{In \autoref{L41} we prove that $\Lambda(\mathfrak{fa}^{-1})$ is of type $\Lambda (\mathfrak{I})$.} Through the isomorphism $\alpha$ the automorphism of $L$ defined by multiplication by $u_i$ corresponds to a matrix $U_i$ that belongs to the intersection $\Gamma (\Lambda(\mathfrak{fa}^{-1}))$ of $\mathrm{Aut}_\mathcal{O} (\Lambda(\mathfrak{fa}^{-1}))$ with $\SL_N (k)$. Moreover, given a prime ideal $\mathfrak{P}$ of $\mathcal{O}_L$ coprime to $\mathfrak{f}$ and $\mathfrak{a}$ and of prime norm $\mathfrak{p} = n (\mathfrak{P})$, the matrices $U_i$ belong to $\Gamma_0(\mathfrak{p},\Lambda(\mathfrak{fa}^{-1}))$.

We denote by $\sigma_1,\ldots,\sigma_N$ the embeddings of $L$ into $\mathbb{C}$ that restrict to the fixed embedding $\sigma:k \to \mathbb{C}$.

\begin{theorem} \label{Theorem2Intro}
Let $p,q$ be non-negative integers, $\mathfrak{f}$ be an ideal of $\mathcal{O}_L$ and let $\mathfrak{a}_1,\ldots,\mathfrak{a}_r$ be integral ideals that form a system of representatives of the ray class group $C_\mathfrak{f}$. Then there exist $v_0 \in k^N$ and two prime ideals $\mathfrak{P}$ and $\tilde{\mathfrak{P}}$ with $\mathfrak{p} = n(\mathfrak{P})$ prime, such that
\begin{multline*}
[U(\mathfrak{f}):U(\mathfrak{f})']\det(\sigma_i(\alpha(e_j))) \zeta_{\mathfrak{f},\mathfrak{P},\tilde{\mathfrak{P}}}^{p,q}(\mathfrak{a},0) \\ = \frac{1}{(p !)^N}  \sum_{\sigma \in S_{N-1}} \mathrm{sgn}(\sigma) \sum_{\substack{x \in \tilde{\mathfrak{P}}^{-1} \mathfrak{f} / \mathfrak{f} \\ x \neq 0}} \langle \mathbf{\Phi}_{\mathfrak{p}}^{pN,qN}(v_0+ \alpha (x),\underline{u}_\sigma,\Lambda (\mathfrak{fa}^{-1})) , (n \circ \alpha^{-1})^p \otimes \overline{(n\circ \alpha^{-1})}^q \rangle
\end{multline*}
for $\mathfrak{a} \in \{\mathfrak{a}_1,\ldots,\mathfrak{a}_r\}$. Here the first sum runs over the symmetric group $S_{N-1}$ of permutations on $N-1$ letters and for $\sigma \in S_{N-1}$ we set $\underline{u}_\sigma=(1,U_{\sigma(1)},U_{\sigma(1)}U_{\sigma(2)},\ldots,U_{\sigma(1)} \cdots U_{\sigma(N-1)})$.
\end{theorem}

Note that each term of the double sum on the right hand side is a generalised Dedekind sum and therefore a polynomial in Kronecker--Eisenstein series evaluated at torsion points on elliptic curves with multiplication by $k$. From \eqref{E:Damerell} we deduce the following corollary. Note that $\phi(\mathfrak{a}_j)\in \overline{\mathbb{Q}}$, so that the algebraicity of $L(\phi,0)$ follows from that of the $\zeta^{-p-1,q}_\mathfrak{f}(\mathfrak{a}_j,0)$.

\begin{corollary} \label{Cintro}
Let $\Omega_\infty$ be any non--zero period of a $\overline{\mathbb{Q}}$-rational global differential on an elliptic curve with complex multiplication by $k$, defined over $\overline{\mathbb{Q}}$. Let $
\phi$ be an algebraic Hecke character of the form \eqref{eq:alg_Hecke_char_intro}. Assume that $p<0$ and $q \geq 0$. Then
$$
L(\phi,0) \in  \Omega_\infty^{N(q-p)}\pi^{-Nq} \overline{\mathbb{Q}}.
$$
\end{corollary}

\medskip
\noindent
{\it Remark.} As was pointed out to us by Don Blasius, one can take $\Omega_\infty$ to be a period of a holomorphic differential on an elliptic curve defined over $k_{ab}$ --- the maximal  abelian extension of $k$. Then
$$
L(\phi,0) \in  \Omega_\infty^{N(q-p)}\pi^{-Nq} k_{ab}E,
$$
where $E$ is the CM field generated by the values of $\phi$. This follows from the fact that the ratio of two arithmetic automorphic functions with Fourier coefficients in $\Q_{ab}$, and having the same weight, takes value in  $k_{ab}$ when evaluated at a CM point.

In fact, one can be more precise: Blasius \cite{Blasius} proves a reciprocity law for values at CM points of modular forms which generalizes that of Shimura for functions. According to it, if a value transforms by a Hecke character, then it is the Deligne period of the motive attached to the Hecke character. Since Theorem \ref{Theorem2Intro} expresses the $L(\phi,0)$ as a linear combination of products of values of $L$-functions of modular Eisenstein series, the general law of Blasius should apply to show the following: Let $M(\phi)$ be the motive --- defined over $L$, with coefficients in $E$, and of rank one --- attached to $\phi$ and let $c^+ \mathrm{Res}_{L/\Q} M(\phi )$ be the period attached by Deligne \cite[\S 8]{Deligne}, we have 
$$L (\mathrm{Res}_{L/\Q} M(\phi )) = c^+ \mathrm{Res}_{L/\Q} M(\phi ) \in (E \otimes \C )^\times / E^\times $$
as conjectured by Deligne \cite{Deligne} as part of a much more general picture. 

%
%

\subsection*{Relation to other works}
In the case $N=2$ \autoref{Theorem1Intro} is proved by Sczech \cite{Sczech84} and Ito \cite{Ito}, in case $(p,q)=(0,0)$, and Obaisi \cite{Obaisi} proved the corresponding \autoref{Theorem2Intro}. In general, partial results towards both \autoref{Theorem2Intro} and \autoref{Cintro} are obtained by Colmez in \cite{Colmez}; see also \cite{FKW} for related works.

\autoref{Cintro} is not new. In the case $L=k$ it is due to Damerell \cite{Damerell}. In the case $N=2$ it is due to Ito \cite{Ito}.  In general it is a particular case of a theorem announced by Harder in \cite{HarderSchappacher,Harder90} that deals with Hecke $L$-functions associated to extensions $L/k$ with $k$ an arbitrary CM fields. When $L=k$ is CM this was known before thanks to works of Shimura \cite{Shimura} and Blasius \cite{Blasius}. Harder provided a proof of his theorem for $N=2$ in \cite{Harder}, but to the authors' knowledge the full details of Harder's proof for $N>2$ have never appeared in print. However the fact that the (regularised) $L$-value of the Hecke character divided by the `Katz period' is algebraic has recently been fully proved by Kings and Sprang \cite{KingsSprang} using completely different techniques that allow them to deduce good integrality results. This generalises works of Shimura and Katz in the case of a CM field to arbitrary extensions of CM fields. In the case where $k$ is a quadratic imaginary field, integrality results of the same quality could be deduced from \autoref{Theorem2Intro} and works of Katz \cite{Katz} showing that certain regularisation (`smoothing') of the expression \eqref{E:Damerell} are algebraic integers. We expect that, in combination with recent work of Andreatta and Iovita \cite{AndreattaIovita}, the explicit formula of \autoref{Theorem2Intro}, conjectured by Sczech and Colmez, could be used to $p$-adically interpolate $L$-values of algebraic Hecke characters of $F$ in the non split case.

Note that, quite similarly as in the work of Kings and Sprang, the cohomology class studied in this paper takes its roots in a certain equivariant cohomology class; we discuss the latter in \cite{French}. The topological origin of this class is enough to give an elementary direct  proof of the integrality of critical values of Hecke L-functions associated to totally real fields, see~\cite{Nori,BKL,Takagi}. 

To conclude let us mention that it is not clear to us if the  formula of \autoref{Theorem2Intro} can be generalised to the case where $k$ is an arbitrary CM field.

\subsection{Notation and conventions} We write $|S|$ for the cardinality of a set $S$. Throughout the paper we fix an integer $N \geq 2$ and let
\begin{equation*}
\begin{split}
V &= \mathbb{C}^N \text{ (column vectors)}.
\end{split}
\end{equation*}
We write $\overline{V}=V \otimes_\mathbb{C} \overline{\mathbb{C}}$ for the complex conjugate of $V$ and $V^\vee$ for the ($\mathbb{C}$-linear) dual of $V$; we identify $V^\vee$ with the space of length $N$ row vectors using the standard dot product. We write $e_1,\ldots,e_N$ for the standard basis of $V$ and $z_1,\ldots,z_N$ for the standard coordinates on $V$ and set $\partial_{z_i}=\partial/\partial z_i$. For a multi-index $I=(i_1,\ldots,i_N) \in \mathbb{Z}_{\geq 0}^N$, we write 
\begin{equation*}
\begin{split}
e^I &=e_1^{i_1} \cdots e_N^{i_N} \in \mathrm{Sym}^N V  \\
z^I &= z_1^{i_1} \cdots z_N^{i_N} \in \mathrm{Sym}^N V^\vee \\
\overline{z}^I &= \overline{z_1}^{i_1} \cdots \overline{z_N}^{i_N} \in \mathrm{Sym}^N \overline{V}^\vee.
\end{split}
\end{equation*}

We denote the transpose of a matrix $X$ by ${^t}X$ and set $X^*={^t}\overline{X}$. We denote by $1_N$ the identity matrix of rank $N$ and by $\mathrm{diag}(t_1,\ldots,t_N)$ a diagonal matrix with diagonal entries $t_1,\ldots,t_N$. Let
\begin{equation*}
\begin{split}
G &= \mathrm{SL}_N(\mathbb{C}) \\
K &= \mathrm{SU}(N) \\
X &= \mathrm{SL}_N(\mathbb{C})/\mathrm{SU}(N).
\end{split}
\end{equation*}

The Lie algebras of $G$ and $K$ are denoted by $\mathfrak{g}$ and $\mathfrak{k}$ respectively.

We write $A^*(X)$ for the space of smooth differential forms on $X$.

Throughout the paper we fix an imaginary quadratic field $k$ and an embedding $\sigma:k \to \mathbb{C}$. We write $\mathrm{N}\mathfrak{p}$ for the norm of a prime ideal $\mathfrak{p}$. We denote by $\mathcal{O}$ the ring of integers of $k$  and define
\begin{equation*}
\begin{split}
V_k &= k^N \text{ (column vectors)} \\
G_k &= \mathrm{SL}_N(k) \\
\mathbf{G} &= \mathrm{Res}_{k/\mathbb{Q}} \mathrm{SL}_{N,k} \\
\end{split}
\end{equation*}

The standard simplex (a simplicial set) is denoted by $\Delta_N$, and its geometric realization by $|\Delta_N|$. We write $\Delta_k * \Delta_r$ for the join of two simplices.

\section{Differential forms on the symmetric space of $\mathrm{SL}_N(\mathbb{C})$}

Fix an integer $N \geq 2$ and let $V=\mathbb{C}^N$ (column vectors). We identify the points of the symmetric space
$$
X:=G/K
$$
of $G=\mathrm{SL}_N(\mathbb{C})$ with positive definite hermitian $N$-by-$N$ matrices $h$ of unit determinant via the map
$$
gK \mapsto h:= {^t}\overline{g}^{-1}\cdot g^{-1};
$$
under this identification the action of $g \in G$ on $X$ by left multiplication corresponds to the action $g \cdot h := {^t}\overline{g}^{-1}h g^{-1}$. A matrix $h \in X$ defines a positive definite hermitian form on $\mathbb{C}^N$ given by $v \mapsto v^*h v$. The entries $h_{ij}$ $(1 \leq i, j \leq N)$ of $h$ define smooth functions $h_{ij}:X \to \mathbb{C}$. 

We write $\mathcal{S}(V)$ for the space of Schwartz functions on $V$.
For $p, q \geq 0$, let
\begin{equation} \label{eq:Vpq_definition}
V^{p,q}=\mathrm{Sym}^p V^\vee \otimes \mathrm{Sym}^q \overline{V};
\end{equation}
it is naturally a representation of $G$. We will identify elements of $V^{p,q}$ with linear functionals on the tensor product of the complex vector spaces $\mathrm{Sym}^p V$ (homogeneous holomorphic polynomials of degree $p$ on $V^\vee$) and $\mathrm{Sym}^q \overline{V}^\vee$ (homogeneous anti-holomorphic polynomials of degree $q$ on $V$).

The natural action of $G$ on $\mathcal{S}(V)$ defined by $(g \cdot f)(v)=f(g^{-1}v)$ turns $\mathcal{S}(V) \otimes V^{p,q}$ into a smooth $G$-module. Let $A^*(X;\mathcal{S}(V) \otimes V^{p,q})$ be the space of differential forms on $X$ valued in $\mathcal{S}(V) \otimes V^{p,q}$. This space carries an action of $G$ given by 
$$
(g , \omega(x,Y)) \mapsto g \cdot \omega(g^{-1}x,g^{-1}Y), \quad x \in X, \ Y \in \land T_x X.
$$
In this section we introduce $G$-invariant differential forms
\begin{equation}
\begin{split}
\psi^{p,q} \in A^{N-1}(X;\mathcal{S}(V) \otimes V^{p,q})^G 
\end{split}
\end{equation}
valued in this $G$-module.

\subsection{Polynomial forms} Fix a vector $v \in V$. We write $(h v)_1,\ldots,(h v)_N$ (resp. $(dh v)_1,\ldots,(dh v)_N)$ for the components of the vector $hv$ (resp. $dh v$):
\begin{equation}
\begin{split}
(hv)_i &= \sum_{1 \leq j \leq N} h_{ij}v_j \in C^\infty(X). \\
(dhv)_i &= \sum_{1 \leq j \leq N} dh_{ij} v_j \in A^1(X).
\end{split}
\end{equation}
Define
\begin{equation} \label{eq:definition_p}
p(v) = 2 (-1)^{N(N-1)/2} \sum_{i \geq 1} (-1)^{i-1} (h v)_i  (dhv)_N \wedge \cdots \wedge \widehat{(dhv)_i} \wedge \cdots \wedge (dhv)_1 \in A^{N-1}(X).
\end{equation}
(Here, as usual, the term under the symbol $\widehat{\quad}$ is to be omitted). Note that, as a function of $v$, $p$ is a (holomorphic) polynomial of degree $N$ and so defines a form
$$
p \in A^{N-1}(X;\mathbb{C}[V]).
$$
The conjugate polynomial $\overline{p(v)}$ defines a form in $A^{N-1}(X;\mathbb{C}[\overline{V}])$. Since $h$ is hermitian, we have
$$
\overline{(hv)_i} = \sum_{j} \overline{h_{ij} v_j} = \sum_j \overline{v_j} h_{ji} = (v^* h)_i
$$
and so we can write
\begin{equation} \label{eq:def_p}
\overline{p(v)} = 2 (-1)^{N(N-1)/2} \sum_{i \geq 1} (-1)^{i-1} (v^*h)_i  (v^* dh)_N \wedge \cdots \wedge \widehat{(v^*dh)_i} \wedge \cdots \wedge (v^* dh)_1.
\end{equation}

\begin{lemma} \label{lemma:p_invariance}
The form $\overline{p}$ is $G$-invariant. That is, for $g \in G$ we have
$$
g^*\overline{p(gv)}=\overline{p(v)}, \qquad v \in V.
$$
\end{lemma}
\begin{proof}
Let us assume that $v \neq 0$ (the case $v=0$ is obvious). The statement then follows from the fact that given a representation of a group $G$ on an $N$-dimensional complex vector space $W$, and a basis $e_1,\ldots,e_N$ of $W$ with dual basis $e_1^\vee,\ldots,e_N^\vee \in W^\vee$, the element $e_1 \otimes e_1^\vee + \cdots + e_N \otimes e_N^\vee$ of $W \otimes W^\vee$ is $G$-invariant.

Namely, consider the $\mathbb{C}$-vector space $W \subset \mathcal{C}^\infty(X)$ spanned by $(v^*h)_{1},\ldots,(v^*h)_{N}$. For $g \in G$ we have
$$
(gv)^*(g\cdot h)=v^*{^t}\overline{g}({^t}\overline{g}^{-1}hg^{-1}) = v^*h g^{-1}.
$$
This shows that $W$ is naturally a representation of $G$ that is isomorphic to the dual $V^\vee$ of $V$. The same statement (with same proof) holds for the $\mathbb{C}$-vector space $\tilde{W} \subset A^1(X)$ spanned by $(v^*dh)_{1},\ldots,(v^*dh)_{N}$.

Consider the map
$$
W \otimes \wedge^{N-1} \tilde{W} \to \wedge^N \tilde{W}, \quad w \otimes \tilde{w} = dw \wedge \tilde{w}.
$$
Here $\wedge^N \tilde{W} \simeq \mathbb{C} \cdot (v^*dh)_{N} \wedge \ldots \wedge (v^*dh)_{1}$ is isomorphic to the trivial $G$-representation via the map $z \cdot (v^*dh)_{N} \wedge \ldots \wedge (v^*dh)_{1} \mapsto z$. Thus we obtain a pairing $W \otimes \wedge^{N-1}\tilde{W} \to \mathbb{C}$. A direct check shows that the basis $(-1)^{N-i}(v^*dh)_{N} \wedge \cdots \wedge \widehat{(v^*dh)_{i}} \wedge \cdots \wedge (v^*dh)_{1}$ ($1 \leq i \leq N$) of $\wedge^{N-1} \tilde{W}$ is dual to the basis $(v^*h)_{i}$ ($1 \leq i \leq N$) of $W$, and the lemma follows.
\end{proof}

\begin{lemma} \label{lemma:p_closed}
Let $v \neq 0$. Then the form 
$$
(v^*hv)^{-N}\overline{p(v)} \in A^{N-1}(X)
$$
is closed.
\end{lemma}
\begin{proof}
An equivalent statement is the equality
\begin{equation} \label{eq:eta_closed_equivalent_form}
N d(v^*hv) \wedge \overline{p(v)} =(v^*hv) d\overline{p(v)}.
\end{equation}
Differentiating \eqref{eq:def_p} we obtain
\begin{equation} \label{eq:dp}
\begin{split}
d\overline{p(v)} &= 2 (-1)^{N(N-1)/2} d \left( \sum_{i \geq 1} (-1)^{i-1} (v^* h)_i  (v^* dh)_N \wedge \cdots \wedge \widehat{(v^*dh)_i} \wedge \cdots \wedge (v^* dh)_1 \right) \\
&= 2 N (-1)^{(N+2)(N-1)/2}  (v^* dh)_N \wedge \cdots \wedge (v^* dh)_1.
\end{split}
\end{equation}
On the other hand we have $d(v^*hv) = \sum_j (v^*dh)_j v_j$ and hence
\begin{multline*}
Nd(v^*hv) \wedge \overline{p(v)} \\
\begin{split}
 &= 2 N (-1)^{N(N-1)/2}  \left(\sum_{j} (v^* dh)_j v_j\right)  \wedge \left(\sum_{i \geq 1} (-1)^{i-1} (v^*h)_i  (v^* dh)_N \wedge \cdots \wedge \widehat{(v^*dh)_i} \wedge \cdots \wedge (v^* dh)_1\right) \\
&= 2 N (-1)^{N(N-1)/2} \sum_j (-1)^{j-1} (v^*dh)_j v_j (v^*h)_j (v^* dh)_N \wedge \cdots \wedge \widehat{(v^*dh)_j} \wedge \cdots \wedge (v^* dh)_1 \\
&= 2 N (-1)^{(N+2)(N-1)/2} \left( \sum_j v_j (v^*h)_j \right) (v^*dh)_N \wedge \cdots \wedge (v^*dh)_1 \\
&= (v^*hv) d\overline{p(v)}.
\end{split}
\end{multline*}
and the assertion follows.
\end{proof}

\subsection{Schwartz forms} We can now define the forms $\psi^{p,q}$ mentioned in the introduction to this section. 
First we consider the case $p=q=0$: for $v \in V$,  we define
\begin{equation}
\psi^{0,0}(v) = e^{-v^*hv} \overline{p(v)}.
\end{equation}

\medskip
\noindent
{\it Remark.} The form $\psi^{0,0}$ arises naturally as a component of a characteristic form defined by Mathai and Quillen. More precisely, the vector bundle $\mathcal{V}=X \times V$ over $X$ with fiber $V$ carries a tautological metric, and the main result of \cite{MathaiQuillen} is the construction of a canonical Thom form $U \in A^{2N}(X \times V)$ and an infinitesimal transgression $\tilde{U}$ of $U$ in $A^{2N-1}(X)$ (denoted $-i_X U_t$ in \cite[\S  7]{MathaiQuillen}). A vector $v \in V$ defines a section of $\mathcal{V}$ over $X$, and the form $\psi^{0,0}(v)$ is essentially obtained from $\tilde{U}$ by contracting with the vector fields $\partial_{z_1}, \ldots,\partial_{z_N}$ (this gives a form in $A^{N-1}(X \times V)$) and then pulling back by this section. We refer to \cite{Takagi} for more details on this perspective.

\medskip

Note that the hermitian form $v \mapsto v^*hv$ on $V$ is positive definite and so $\psi^{0,0}$, as a function of $v$, belongs to the Schwartz space $\mathcal{S}(V)$. Also note that, for any $g \in G$, the expression $v^*hv$ is invariant upon replacing $h$ with $g^*h$ and $v$ with $gv$, and so \autoref{lemma:p_invariance} implies that $\psi^{0,0}$ is $G$-invariant:
\begin{equation} \label{eq:psi_0_invariance}
g^*\psi^{0,0}(gv)=\psi^{0,0}(v), \qquad g \in G.
\end{equation}
Thus $\psi^{0,0} \in A^{N-1}(X;\mathcal{S}(V))^G$.

For arbitrary $p,q \geq 0$ we define
$$
\psi^{p,q} \in A^{N-1}(X;\mathcal{S}(V)\otimes V^{p,q})^G
$$
so that its value on $P \otimes \overline{Q}$, where $P$ (resp. $Q$) is a holomorphic polynomial of degree $p$ on $V^\vee$ (resp. a holomorphic polynomial of degree $q$ on $V$), is given by
\begin{equation}
\begin{split}
\psi^{p,q}(v,P \otimes \overline{Q})= \overline{Q}(\overline{v})P(-\partial_{z_1},\ldots,-\partial_{z_N}) \psi^{0,0}(v)
\end{split}
\end{equation}
From now on we often omit the indices $p,q$ and simply write $\psi(v,P \otimes \overline{Q})$.
One can give a more explicit expression for $\psi(v,P \otimes \overline{Q})$: the identity
\begin{equation}
-\partial_{z_i} \left( e^{-v^*hv} \right)=(v^*h)_i e^{-v^*hv}
\end{equation} 
gives
\begin{equation}
P(-\partial_{z_1},\ldots,-\partial_{z_N})  \left( e^{-v^*hv} \right) = P((v^*h)_1,\ldots,(v^*h)_N) e^{-v^*hv};
\end{equation}
since if $\overline{p(v)}$ is an anti-holomorphic polynomial we have $\partial_{z_i}\overline{p(v)} = 0$ for all $i$, and we conclude that
\begin{equation}
\psi(v,P \otimes \overline{Q}) = e^{-v^*hv} \overline{p(v,P,Q)},
\end{equation}
with
\begin{equation}
\overline{p(v,P,Q)} := \overline{Q(v)}P(v^*h)\overline{p(v)}.
\end{equation}
Note that $\overline{p(\cdot,P,Q)}$ is an anti-holomorphic polynomial in $v$ of degree $N+p+q$. This expression shows that, generalizing the invariance property \eqref{eq:psi_0_invariance}, we have
\begin{equation} \label{eq:psi_invariance_with_coeffs}
\begin{split}
g^*\psi(gv, gP \otimes \overline{gQ}) &= \psi(v,P \otimes \overline{Q}).
\end{split}
\end{equation}
The following is a generalization of \autoref{lemma:p_closed}.

\begin{lemma} \label{lemma:p_closed_coeffs}
Let $v \neq 0$. For any $P \in \mathrm{Sym}^p V$ and $Q \in \mathrm{Sym}^q V^\vee$, the form
$$
(v^*hv)^{-N-p} \overline{p(v,P,Q)} \in A^{N-1}(X)
$$
is closed.
\end{lemma}
\begin{proof}
Since $d\overline{Q(v)}=0$, it suffices to assume that $Q=1$ and that $P$ is monomial, say $P=e^I$ for some multi-index $I$ of degree $p$. Then
$P(v^*h)=(v^*h)_1^{i_1} \cdots (v^*h)_N^{i_N}$ and 
$$
dP(v^*h) = \left( \sum_{j} i_j \frac{d(v^*h)_j}{(v^*h)_j} \right)P(v^*h),
$$
and so
\begin{multline*}
dP(v^*h) \wedge \overline{p(v)}  \\
\begin{split}
&= 2 (-1)^{N(N-1)/2} P(v^*h)  \left( \sum_{j} i_j \frac{d(v^*h)_j}{(v^*h)_j} \right) \wedge \left( \sum_{i \geq 1} (-1)^{i-1} (v^*h)_i  (v^* dh)_N \wedge \cdots \wedge \widehat{(v^*dh)_i} \wedge \cdots \wedge (v^* dh)_1 \right) \\
&= 2 (-1)^{N(N-1)/2} P(v^*h) \sum_{j} i_j (-1)^{j-1} (v^*dh)_j \wedge (v^* dh)_N \wedge \cdots \wedge \widehat{(v^*dh)_j} \wedge \cdots \wedge (v^* dh)_1 \\
&= 2 (-1)^{(N+2)(N-1)/2} \left(\sum_j i_j \right) P(v^*h)  (v^* dh)_N \wedge \cdots \wedge  (v^* dh)_1 \\
&= p N^{-1} P(v^*h) d\overline{p(v)},
\end{split}
\end{multline*}
where the last equality follows from \eqref{eq:dp}. Using \eqref{eq:eta_closed_equivalent_form} we compute
\begin{multline*}
d((v^*hv)^{-N-p}P(v^*h)\overline{p(v)}) \\
\begin{split}
&= (v^*hv)^{-N-p-1} \left[-(N+p)d(v^*hv) \wedge P(v^*h) \overline{p(v)}  + (v^*hv)dP(v^*h) \wedge \overline{p(v)} + (v^*hv)P(v^*h)d\overline{p(v)} \ \right] \\
&= (v^*hv)^{-N-p} \left[-(N+p)N^{-1}+pN^{-1}+1 \right]P(v^*h)d\overline{p(v)} \\
&= 0.
\end{split}
\end{multline*}
\end{proof}

\subsection{Mellin transform} 
We define $\eta(v,s)$ to be the Mellin transform of $\psi(v)$; that is, for holomorphic polynomials $P$ and $Q$ define
\begin{equation} \label{def:eta_Mellin_transform}
\eta(v,P \otimes \overline{Q},s) = \int_0^\infty \psi(tv,P \otimes \overline{Q}) t^{s+N+p-q} \frac{dt}{t}.
\end{equation}
Then 
\begin{equation} \label{eq:eta_invariance}
g^* \eta(gv,gP \otimes \overline{gQ},s) = \eta(v,P \otimes \overline{Q},s), \quad g \in G
\end{equation}
because $\psi$ is $G$-invariant. Since $\overline{p(tv,P,Q)} = t^{N+p+q} \overline{p(v,P,Q)}$, we have
\begin{equation} \label{eq:eta_explicit}
\begin{split}
\eta(v,P \otimes \overline{Q},s) &=  \int_0^\infty e^{-t^2v^* h v} t^{s+2N+2p} \frac{dt}{t} \overline{p(v,P,Q)} \\
&=  2^{-1}\Gamma(N+p+s/2) (v^*hv)^{-s/2-N-p}\overline{p(v,P,Q)}.
\end{split}
\end{equation}

\begin{lemma} \label{lemma:d_eta_explicit}
We have
\[
d\eta(v,P \otimes \overline{Q},s) = c(s) (v^*hv)^{-s/2-N-p} \overline{Q(v)}P(v^*h)d\overline{p(v)},
\]
where $c(s)=(-4N)^{-1}s\Gamma(N+p+s/2)$.
\end{lemma}
\begin{proof}
By \autoref{lemma:p_closed_coeffs} we have $d((v^*hv)^{-N-p} \overline{p(v,P,Q)})=0$. Using \eqref{eq:eta_closed_equivalent_form}, we compute
\begin{equation*}
\begin{split}
2 \Gamma(N+p+s/2)^{-1} & d\eta(v,P \otimes \overline{Q},s) \\ &= d((v^*hv)^{-s/2} (v^*hv)^{-N-p} \overline{p(v,P,Q)}) \\
&= -2^{-1}s (v^*hv)^{-1-s/2} d(v^*h v) \wedge (v^*hv)^{-N-p} \overline{p(v,P,Q)} \\
&= -(2N)^{-1} s (v^*hv)^{-s/2-N-p} \overline{Q(v)}P(v^*h)d\overline{p(v)}.
\end{split}
\end{equation*}
\end{proof}

Using \autoref{lemma:d_eta_explicit} we can represent the form $d\eta(v,P\otimes \overline{Q},s)$ as a Mellin transform: we define 
$$
\phi \in A^N(X;\mathcal{S}(V) \otimes V^{p,q})^G
$$
by
\begin{equation} \label{eq:def_varphi}
\phi(v,P \otimes \overline{Q}) = e^{-v^*hv} \overline{Q(v)} P(v^*h)d\overline{p(v)};
\end{equation} 
then  the above lemma implies that
\begin{equation} \label{eq:d_eta_Mellin}
d\eta(v,P\otimes \overline{Q},s) =-\frac{s}{2N} \int_0^\infty \phi(tv,P \otimes \overline{Q}) t^{s+N+p-q} \frac{dt}{t}.
\end{equation}
For further reference we note the homogeneity property (which follows from \eqref{eq:eta_explicit}):
\begin{equation} \label{eta_homogeneity}
\eta(zv,P(z\cdot) \otimes \overline{Q(z^{-1} \cdot)},s) = |z|^{-s} z^{-N} \eta(v,P \otimes \overline{Q},s)
\end{equation}
for $z \in \mathbb{C}^\times$.

\subsection{Example: the case $N=2$}
We compute the form $\psi^{0,0}(v)$ when $N=2$. We have
\begin{equation}
\begin{split}
\psi^{0,0}(v) &= -2 e^{-v^*hv} (\overline{(hv)}_1 (\overline{dhv})_2-\overline{(hv)}_2 (\overline{dhv})_1) \\
&= - 2 e^{-v^*hv} (\omega_{11}\overline{v_1}^2 +\omega_{12}\overline{v_1 v_2} + \omega_{22} \overline{v_2}^2), 
\end{split}
\end{equation}
with
\begin{equation}
\begin{split}
\omega_{11}&=\overline{h_{11}}d\overline{h_{21}}-\overline{h_{21}}d{\overline{h_{11}}}=h_{11}dh_{12}-h_{12}dh_{11}, \\
\omega_{12}&=h_{11}dh_{22}-h_{12}dh_{21}+h_{21}dh_{12}-h_{22}dh_{11}, \\
\omega_{22}&=h_{21}dh_{22}-h_{22}dh_{21}.
\end{split}
\end{equation}
Let us rewrite the expression in classical coordinates. For $\tau=(z,y) \in \mathbb{H}^3=\mathbb{C} \times \mathbb{R}_{>0}$, write $g_\tau=\left(\begin{smallmatrix} y^{1/2} & zy^{-1/2} \\ 0 & y^{-1/2}\end{smallmatrix}\right)$. The map $\tau \mapsto g_{\tau}K$ identifies $\mathbb{H}^3$ with $X=\mathrm{SL}_2(\mathbb{C})/\mathrm{SU}(2)$. In these coordinates we have
\[
h_\tau = {^t}\overline{g_{\tau}}^{-1}g_\tau^{-1}  = \begin{pmatrix} y^{-1/2} & 0 \\ -\overline{z} y^{-1/2} & y^{1/2} \end{pmatrix} \begin{pmatrix} y^{-1/2} & -zy^{-1/2}\\ 0 & y^{1/2} \end{pmatrix} = y^{-1} \begin{pmatrix} 1 & -z \\ -\overline{z} & y^2+|z|^2 \end{pmatrix},
\]
and
\[
v^*h_\tau v = |g_\tau^{-1}v|^2 = y^{-1}(|v_1-zv_2|^2+|yv_2|^2).
\]
Hence
\begin{equation}
\begin{split}
dh_\tau &= -y^{-2}dy \begin{pmatrix} 1 & -z \\ -\overline{z} & y^2+|z|^2 \end{pmatrix}+ y^{-1} d\begin{pmatrix} 1 & -z \\ -\overline{z} & y^2+|z|^2 \end{pmatrix} \\
&= -h_\tau y^{-1}dy + y^{-1} \begin{pmatrix} 0 & -dz \\ -d\overline{z} & 2ydy+zd\overline{z}+\overline{z}dz \end{pmatrix} 
\end{split}
\end{equation}
and we compute
\begin{equation}
\begin{split}
\omega_{11} &=  y^{-2}dz, \\
\omega_{12} &= -2(y^{-1}dy +y^{-2}\overline{z}dz), \\
\omega_{22} &= 2\overline{z}y^{-1}dy-d\overline{z}+\overline{z}^2 y^{-2}dz.
\end{split}
\end{equation}

Writing $\psi^{0,0}(v)=\psi(v)_y dy+\psi(v)_z dz+\psi(v)_{\overline{z}} d\overline{z}$ we obtain
\begin{equation}
\begin{split}
\psi(v)_{\overline{z}} &= 2 e^{-v^*h_\tau v} \cdot \overline{v_2}^2 \\
\psi(v)_{y} &= -2 e^{-v^*h_\tau v} (-2y^{-1}\overline{v_1v_2}+2\overline{z}y^{-1}\overline{v_2}^2) \\
&= 4y^{-1}e^{-v^*h_\tau v}\overline{(v_1-zv_2)v_2} \\
\psi(v)_{z} &= -2 y^{-2}e^{-v^*h_\tau v} (\overline{v_1}^2-2\overline{z} \overline{v_1v_2}+\overline{z}^2\overline{v_2}^2) \\
&= -2y^{-2}e^{-v^*h_\tau v} \overline{(v_1-zv_2)}^2.
\end{split}
\end{equation}
The Mellin transform $\eta^{0,0}(v,s)=\eta(v,s)_y dy+\eta(v,s)_z dz+\eta(v,s)_{\overline{z}} d\overline{z}$ (defined in \eqref{def:eta_Mellin_transform} below) is then given by:
\begin{equation} \label{eq:eta_explicit_formula_N_equals_2}
\begin{split}
\eta(v,s)_{\overline{z}} &= \Gamma(s/2+2)y^{s/2} \frac{(\overline{yv_2})^2}{(|v_1-zv_2|^2+|yv_2|^2)^{s/2+2}} \\
\eta(v,s)_{y} &= 2\Gamma(s/2+2)y^{s/2} \frac{\overline{(v_1-zv_2)yv_2}}{(|v_1-zv_2|^2+|yv_2|^2)^{s/2+2}} \\
\eta(v,s)_{z} &= - \Gamma(s/2+2)y^{s/2} \frac{\overline{(v_1-zv_2)}^2}{(|v_1-zv_2|^2+|yv_2|^2)^{s/2+2}}
\end{split}
\end{equation}
Thus we recover the form introduced by Ito in \cite{Ito}.

\subsection{Fourier transform} Recall that the Cartan decomposition $\mathfrak{g}=\mathfrak{p} \oplus \mathfrak{k}$ identifies the tangent space $T_{eK} X$ at the point $eK \in X$ with $\mathfrak{p}$. Given $Y \in \wedge^{N-1} \mathfrak{p}$ and polynomials $P$ and $Q$, evaluation at $Y$ defines a Schwartz function
$$
\psi(Y,P \otimes \overline{Q}) \in \mathcal{S}(V)
$$
given explicitly by
$$
\psi(v,Y;P\otimes \overline{Q}) = e^{-v^*v}\overline{p(v,Y;P,Q)},
$$
with $\overline{p(v,Y;P,Q)}=\overline{Q(v)}P(v^*)\overline{p(v,Y)}$. 

We write $\langle \cdot, \cdot \rangle$ for the scalar product on $V$ defined by $\langle v,w \rangle=2 \mathrm{Re}(w^*v)$ and given a Schwartz form $f \in \mathcal{S}(V)$, we define its Fourier transform $\mathcal{F}f \in \mathcal{S}(V)$ by
$$
\mathcal{F}f(v)=\int_V f(w)e^{2\pi i \langle v,w \rangle} dw, 
$$
where $dw$ denotes the Lebesgue measure on $\mathbb{C}^N$.  Since the polynomial $\overline{p(v,Y;P,Q)}$ is anti-holomorphic, it is also harmonic and hence we have
\begin{equation} \label{eq:psi_Four_transform}
\mathcal{F}\psi(Y,P \otimes \overline{Q}) = C \psi(Y,P \otimes \overline{Q})
\end{equation}
for some constant $C$ satisfying $C^4 =1$. In particular, $\mathcal{F}\psi(0,Y;P \otimes \overline{Q})=\psi(0,Y;P \otimes \overline{Q})=0$. Similar statements hold for $\phi(Y,P\otimes \overline{Q})$ for any $Y \in \wedge^N \mathfrak{p}$.

\subsection{Integral on a maximal torus} \label{subsection:Integral_on_max_torus} Let $T \subset G$ be the torus of diagonal matrices. The inclusion of $T$ in $G$ induces an embedding
$$
T/T \cap K \rightarrow X
$$
identifying $T/T \cap K$ with the submanifold of $X$ consisting of diagonal hermitian matrices. This submanifold is diffeomorphic to $\mathbb{R}_{>0}^{N-1}$: writing
\begin{equation} \label{eq:C_manifold}
C=\left\{ (t_1,\ldots,t_N) \in \mathbb{R}_{>0}^N \ | \ t_1\cdots t_N=1 \right\},
\end{equation}
the map $(t_1,\ldots,t_N) \mapsto \mathrm{diag}(t_1^{-1},\ldots, t_N^{-1})T \cap K$ identifies $C \simeq T/T\cap K$. We use this identification to orient $T/T \cap K$ as follows: forgetting the coordinate $t_N$ gives a diffeomorphism $C \simeq \mathbb{R}_{>0}^{N-1}$. We orient $C$, and hence $T/T\cap K$, by pulling back the standard orientation of $\mathbb{R}_{>0}^{N-1}$ (given by the volume form $\tfrac{dt_1}{t_1} \wedge \cdots \wedge \tfrac{dt_{N-1}}{t_{N-1}}$).

\begin{lemma} \label{lemma:eta_integration} Let $v \neq 0$. If $\mathrm{Re}(s)+2N+2p>0$, then the form $\eta(v,P \otimes \overline{Q},s)$ is integrable on $T/T\cap K$. For $P=e^I$ and $Q=z^J$ monomial with $I=(i_1,\ldots,i_N)$ and $J=(j_1,\ldots,j_N)$ multi-indices, we have
$$
\int_{T/T\cap K} \eta(v,P \otimes \overline{Q},s) =  \prod_{k=1}^N \Gamma\left( \frac{s}{2N}+1+i_k \right) \frac{\overline{v_k}^{j_k}}{|v_k|^{s/N}v_k^{i_k+1}}.
$$
\end{lemma}
\begin{proof}
Since $\eta(v,P \otimes \overline{Q},s)=\overline{Q(v)} \eta(v,P \otimes 1,s)$, we may assume that $Q=1$. In the above coordinates for $C$ we have $h=\mathrm{diag}(t_1^2,\ldots,t_N^2)$ and $dh=2\mathrm{diag}(t_1 dt_1,\ldots,t_N dt_N)$, and so the restriction of $P(v^*h)\overline{p(v)}$ to $T/T \cap K$ is given by
\begin{multline*}
\prod_{j=1}^N (t_j^2 \overline{v_j})^{i_j} \cdot 2 (-1)^{N(N-1)/2} \sum_{j} (-1)^{j-1} \overline{v_j}t_j^2 (\overline{v_N}2t_N dt_N) \wedge \cdots \wedge \widehat{(\overline{v_j}2t_j dt_j)} \wedge \cdots \wedge (\overline{v_1}2t_1 dt_1) \\
= 2^N (-1)^{N(N-1)/2} \prod_{j=1}^N t_j^{2i_j} \overline{v_j}^{i_j+1} \sum_j (-1)^{j-1} \frac{dt_N}{t_N} \wedge \cdots \wedge \widehat{\frac{dt_j}{t_j}} \wedge \cdots \wedge \frac{dt_1}{t_1} \qquad (\text{since } t_1\cdots t_N=1).
\end{multline*}
For $t \in C$ and $u >0$, set $u_i = t_i u$. This gives $u_1 \cdots u_N=u^N$ and $\tfrac{du_i}{u_i}=\tfrac{dt_i}{t_i}+\tfrac{du}{u}$, and hence
\[
\left(\sum_j (-1)^{j-1} \frac{dt_N}{t_N} \wedge \ldots \wedge \widehat{\frac{dt_j}{t_j}} \wedge \ldots \wedge \frac{dt_1}{t_1}\right) \wedge \frac{du}{u}  = \frac{du_N}{u_N} \wedge \ldots \wedge \frac{du_1}{u_1}.
\]
The map $((t_1,\ldots,t_N),u) \mapsto (u_1,\ldots,u_N)$ induces a diffeomorphism $C \times \mathbb{R}_{>0} \simeq \mathbb{R}_{>0}^N$. Using this as change of variables, we compute
\begin{equation*}
\begin{split}
\int_{T/T \cap K} \eta(v,P \otimes 1,s) &= \int_{C} \int_0^\infty \psi(uv,P \otimes 1)u^{s+N+p} \frac{du}{u} \\
&=  2^N (-1)^{N(N-1)/2} \int_{\mathbb{R}_{>0}^N} e^{-\sum u_j^2|v_j|^2}  \left(\prod_{j=1}^N t_j^{2i_j} u^{i_j+1} \overline{v_j}^{i_j+1} \right) u^{s+N+p} \frac{du_N}{u_N} \wedge \ldots \wedge \frac{du_1}{u_1} \\
&= 2^N \prod_{j=1}^N \overline{v_j}^{i_j+1} \int_0^\infty e^{-u_j^2|v_j|^2} u_j^{(s+2N)/N+2i_j} \frac{du_j}{u_j} \\
&=   \prod_{j=1}^N \Gamma\left( \frac{s}{2N}+1+i_j \right) \frac{\overline{v_j}^{i_j+1}}{|v_j|^{s/N+2+2i_j}}.
\end{split}
\end{equation*}
\end{proof}

The above lemma shows that the integral of $\eta(v,P \otimes \overline{Q},s)$ on $T/T \cap K$ has meromorphic continuation to $s \in \mathbb{C}$ that is regular at $s=0$. Its value at $s=0$ for $P=e^I$ and $Q=z^J$ is 
\begin{equation} \label{eq:int_eta_at_0_expression_1}
\left. \int_{T/T \cap K} \eta(v,P \otimes \overline{Q},s) \right|_{s=0} =  \prod_{k=1}^N i_k! \frac{\overline{v_k}^{j_k}}{v_k^{i_k+1}}.
\end{equation}
It follows easily that for arbitrary $P$ we can write
$$
\int_{T/T \cap K} \eta(v,P \otimes \overline{Q},s) = C(s)  \overline{Q(v)} P(-\partial_{z_1},\ldots,-\partial_{z_N}) \left( \prod_{j=1}^N \frac{\overline{v_j}}{|v_j|^{s/N+2}} \right),
$$
for some meromorphic function $C(s)$ such that $C(0)=1$.

\section{Eisenstein cocycle}

Let $k$ be an imaginary quadratic field with ring of integers $\mathcal{O}$. We fix an integer $N \geq 2$ and let $V_k=k^N$ and $G_k=\mathrm{SL}_N(k)$ (recall that $V=\mathbb{C}^N$ and $G=\mathrm{SL}_N(\mathbb{C})$). We also fix an embedding $\sigma:k \rightarrow \mathbb{C}$, which makes $V$ a $k$-module and induces inclusions $V_k \subset V$ and $G_k \subset G$. 

Given a non-zero ideal $\mathfrak{I}$ of $\mathcal{O}$, define 
\begin{equation} \label{eq:def_Lambda_I}
\Lambda(\mathfrak{I}) = \mathfrak{I}^{-1} \oplus \mathcal{O}^{N-1}.
\end{equation}
It is an $\mathcal{O}$-submodule of $k^N$ that we regard as a lattice in $\mathbb{C}^N$ via the embedding $k^N \to \mathbb{C}^N$ induced by $\sigma$. We write $\Gamma(\Lambda(\mathfrak{I}))$ for the intersection of $\mathrm{Aut}_\mathcal{O}(\Lambda(\mathfrak{I}))$ with $\mathrm{SL}_N(k)$; more explicitly
\begin{equation}
\Gamma(\Lambda(\mathfrak{I}))= \left\{ \left. \begin{pmatrix} a & {^t}b \\ c & D \end{pmatrix} \in \mathrm{SL}_N(k)\right| a \in \mathcal{O}, D \in M_{N-1}(\mathcal{O}), b \in (\mathfrak{I}^{-1})^{N-1}, c \in \mathfrak{I}^{N-1} \right\}.
\end{equation}

Let $\mathfrak{p} \subset \mathcal{O}$ be a prime ideal coprime to $\mathfrak{I}$. We define a congruence subgroup $\Gamma_0(\mathfrak{p},\Lambda(\mathfrak{I}))$ of $\Gamma(\Lambda(\mathfrak{I}))$ by
\begin{equation}
\Gamma_0(\mathfrak{p},\Lambda(\mathfrak{I}))= \left\{ \left. \begin{pmatrix} a & {^t}b \\ c & D \end{pmatrix} \in \Gamma(\Lambda(\mathfrak{I})) \ \right| \ c \in (\mathfrak{p}\mathfrak{I})^{N-1} \right\};
\end{equation}
thus $\Gamma_0(\mathfrak{p},\Lambda(\mathfrak{I})) = \Gamma(\Lambda(\mathfrak{pI})) \cap \Gamma(\Lambda(\mathfrak{I}))$. When $\mathfrak{I}=\mathcal{O}$ we have $\Gamma(\mathcal{O})=\mathrm{SL}_N(\mathcal{O})$ and $\Gamma_0(\mathfrak{p},\Lambda(\mathcal{O})) = \Gamma_0(\mathfrak{p})$ is the standard level $\mathfrak{p}$ subgroup of $\mathrm{SL}_N(\mathcal{O})$.

In this section we prove \autoref{Theorem1Intro}. We first define a more general cocycle 
\begin{equation}
\mathbf{\Phi}^{p,q}_\mathfrak{p} (\Lambda(\mathfrak{I})) : \Gamma_0(\mathfrak{p},\Lambda(\mathfrak{I}))^N \to \mathcal{F} \otimes V^{p,q},
\end{equation}
where $V^{p,q}$ is given in \eqref{eq:Vpq_definition} and $\mathcal{F}$ is a certain space of functions defined on complements of unions of affine hyperplanes in $V$, endowed with a natural action of $\mathrm{SL}_N(k)$. In the last section we will show that its cohomology class is non-trivial by computing its value explicitly on the units of degree $N$ field extensions of $k$.

\subsection{Definition of the cocycle} Let $\mathfrak{I} \subset k$ be a fractional ideal. Then $\sigma(\mathfrak{I}) \subset \mathbb{C}$ is a lattice.
Given a pair of non-negative integers $p,q \in \mathbb{Z}_{\geq 0}$ and $z \in \mathbb{C}$, define the Kronecker--Eisenstein series
\begin{equation}
K^{p,q}(z,\mathfrak{I},s) = p! \sum_{a \in \sigma(\mathfrak{I})}  \frac{\overline{z+a}^q}{(z+a)^{p+1}|z+a|^{s}}, \qquad z \notin \sigma(\mathfrak{I}).
\end{equation}
The sum converges absolutely for $\mathrm{Re}(s)>1+q-p$ and for $z$ in a compact subset of $\mathbb{C}$. The series $K^{p,q}(z,\mathfrak{a},s)$ has an analytic continuation to the whole $s$-plane that is regular at $s=0$, see e.g. \cite{WeilEllipticFunctions,Colmez,ColmezSchneps}.

More generally, for an $\mathcal{O}$-lattice $\Lambda \subset k^N$, 
let $U(\Lambda)$ be the open subset of $\mathbb{C}^N$ obtained by removing all translates of coordinate hyperplanes by $\lambda \in \sigma(\Lambda)$. For $I=(i_1,\ldots,i_N)$ and $J=(j_1,\ldots,j_N)$ in $\mathbb{Z}_{\geq 0}^N$ and $z \in U(\Lambda)$, define
\begin{equation}
K^{I,J}(z,\Lambda,s) = \sum_{\lambda \in \sigma(\Lambda)} \prod_{1 \leq k \leq N} i_k! \frac{\overline{z_k+\lambda_k}^{j_k}}{(z_k+\lambda_k)^{i_k+1}|z_k+\lambda_k|^{s}}.
\end{equation}
The function $K^{I,J}(z,\Lambda,s)$ can be expressed as a homogeneous degree $N$ polynomial of Kronecker--Eisenstein series: pick non-zero fractional ideals $\mathfrak{I}_1,\ldots,\mathfrak{I}_N$ of $k$ such that $\Lambda \supseteq \mathfrak{I}_1 \oplus \cdots \oplus \mathfrak{I}_N$. Then
\begin{equation}
K^{I,J}(z,\Lambda,s) = \sum_{\lambda \in \Lambda / \mathfrak{I}_1 \oplus \cdots \oplus \mathfrak{I}_N}  K^{I,J}(z+\sigma(\lambda),\mathfrak{I}_1 \oplus \cdots \oplus \mathfrak{I}_N,s),
\end{equation}
and
\begin{equation}
K^{I,J}(z,\mathfrak{I}_1 \oplus \cdots \oplus \mathfrak{I}_N,s) = \prod_{1 \leq k \leq N} K^{i_k,j_k}(z_k,\mathfrak{I}_k,s).
\end{equation}
Thus $K^{I,J}(z,\Lambda,s)$ converges absolutely for $\mathrm{Re}(s)>1+\max\{j_k-i_k\}$ and has analytic continuation to all $s\in \mathbb{C}$ that is regular at $s=0$. We set
\begin{equation}
K^{I,J}(z,\Lambda) = K^{I,J}(z,\Lambda,0)
\end{equation}
and define
$$
\mathcal{F} := \mathrm{span}\langle K^{I,J}(\gamma^{-1}z,\Lambda) \ | \  \gamma \in \mathrm{SL}_N(k), \Lambda \text{ an } \mathcal{O}\text{-lattice in } k^N\rangle,
$$
which carries a natural action of $\mathrm{SL}_N(k)$.

\begin{definition} Let $A$ be a matrix in $\mathrm{End}(\Lambda(\mathfrak{I})) \cap \mathrm{GL}_N(k)$. Then $A^{-1}\Lambda(\mathfrak{I}) \supseteq \Lambda(\mathfrak{I})$ and we define the generalized Dedekind sum
\begin{equation*}
\begin{split}
D^{I,J}(z,A,\Lambda(\mathfrak{I})) &= \det A^{-1} K^{I,J}(A^{-1}z,A^{-1}\Lambda(\mathfrak{I})) \\
& = \det A^{-1} \sum_{\lambda \in \Lambda(\mathfrak{I})/A\Lambda(\mathfrak{I})} K^{I,J}(A^{-1}(z+\sigma(\lambda)),\Lambda(\mathfrak{I})).
\end{split}
\end{equation*}
Let $\mathfrak{p}$ be a proper ideal of $\mathcal{O}$ coprime to $\mathfrak{I}$ and $\mathrm{N}\mathfrak{p}$ be its norm. Define
\begin{equation*}
D^{I,J}_{\mathfrak{p}}(z,A,\Lambda(\mathfrak{I})) = D^{I,J}(z,A,\Lambda(\mathfrak{pI}))-\mathrm{N}\mathfrak{p} \cdot D^{I,J}(z,A,\Lambda(\mathfrak{I})).
\end{equation*}
If $A \in \mathrm{End}(\Lambda(\mathfrak{I}))$ but $A$ is not invertible, set
\begin{equation*}
D^{I,J}(z,A,\Lambda(\mathfrak{I})) = D^{I,J}_{\mathfrak{p}}(z,A,\Lambda(\mathfrak{I})) = 0.
\end{equation*}
\end{definition}

For $p, q \in \mathbb{Z}_{\geq 0}$, recall the $G$-representation $V^{p,q}$ introduced in \eqref{eq:Vpq_definition}. A basis of $V^{p,q}$ is given by the vectors
\begin{equation}
e^{I,J} :=(({^t}e_1)^{i_1} \cdots ({^t}e_N)^{i_N}) \otimes (\overline{e_1}^{j_1} \cdots \overline{e_N}^{j_N}),
\end{equation}
where $I,J \in \mathbb{Z}_{\geq 0}^{N}$ satisfy $i_1 + \ldots +i_N = p$ and $j_1 + \ldots + j_N = q$.

Recall that given a group $\Gamma$ and a $\mathbb{Z}[\Gamma]$-module $M$, a map $\alpha : \Gamma^N \to M$ is said to be a homogeneous $(N-1)$-cocycle if it is equivariant, that is,
\begin{equation} \label{eq:cocycle_def_equiv}
\alpha(\gamma \gamma_1,\ldots,\gamma \gamma_N)=\gamma \alpha(\gamma_1,\ldots,\gamma_N), \qquad \gamma,\gamma_1,\ldots,\gamma_N \in \Gamma,
\end{equation}
and satisfies
\begin{equation} \label{eq:cocycle_def_closed}
\sum_{1 \leq i \leq N+1} (-1)^{i-1} \alpha(\gamma_1,\ldots,\gamma_{i-1},\gamma_{i+1},\ldots,\gamma_{N+1}) = 0, \qquad \gamma_1,\ldots,\gamma_{N+1} \in \Gamma.
\end{equation}

\begin{theorem}
Let $\mathfrak{I}, \mathfrak{p} \subseteq \mathcal{O}$ be non-zero coprime ideals of $\mathcal{O}$ and assume that $\mathfrak{p} \neq \mathcal{O}$. Given $\underline{\gamma}=(\gamma_1,\ldots,\gamma_N) \in \Gamma_0(\mathfrak{p},\Lambda(\mathfrak{I}))^N$, let
\begin{equation*}
A(\underline{\gamma}) = (\gamma_1 e_1 | \cdots | \gamma_N e_1) \in M_N(\mathcal{O})
\end{equation*}
be the matrix formed by the first columns of $\gamma_1,\ldots,\gamma_N$. Then $A(\underline{\gamma}) \in \mathrm{End}(\Lambda(\mathfrak{I}))$. For fixed $p,q$ in $\mathbb{Z}_{\geq 0}$, define a map
\begin{equation*}
\mathbf{\Phi}^{p,q}_\mathfrak{p} (\Lambda(\mathfrak{I})):\Gamma_0(\mathfrak{p},\Lambda(\mathfrak{I}))^N \to \mathcal{F} \otimes V^{p,q}
\end{equation*}
by
\begin{equation*}
\mathbf{\Phi}^{p,q}_\mathfrak{p} (z,\underline{\gamma},\Lambda(\mathfrak{I})) = \sum_{\substack{|I|=p \\ |J|=q}} D^{I,J}_{\mathfrak{p}}(z,A(\underline{\gamma}),\Lambda(\mathfrak{I})) \otimes A(\underline{\gamma}) e^{I,J}.
\end{equation*}
Then $\mathbf{\Phi}^{p,q}_\mathfrak{p} (\Lambda(\mathfrak{I}))$ is a homogeneous $(N-1)$-cocycle.
\end{theorem}

Note that the first row of the matrix $A(\underline{\gamma})$ in the statement has entries in $\mathcal{O}$ whereas all its other rows have entries in $\mathfrak{I}$. The statement that $A(\underline{\gamma}) \in \mathrm{End}(\Lambda(\mathfrak{I}))$ follows. Note also that
\begin{equation}
A(\gamma \gamma_1,\ldots,\gamma \gamma_N)=\gamma A(\gamma_1,\ldots,\gamma_N), \qquad \gamma,\gamma_1,\ldots,\gamma_N \in \Gamma_0(\mathfrak{p},\Lambda(\mathfrak{I})).
\end{equation}
The equivariance property \eqref{eq:cocycle_def_equiv} of $\mathbf{\Phi}^{p,q}_\mathfrak{p} (\Lambda(\mathfrak{I}))$ follows from this. Thus it remains to show the cocycle property \eqref{eq:cocycle_def_closed}. To prove it we will next define --- as an Eisenstein series --- a closed $\Gamma_0(\mathfrak{p},\Lambda(\mathfrak{I}))$-invariant differential form
\begin{equation}
E_{\mathfrak{p}}(z,\psi^{p,q},\Lambda(\mathfrak{I})) \in A^{N-1}(X) \otimes V^{p,q}
\end{equation}
and $(N-1)$-dimensional submanifolds
\begin{equation}
\Delta(\underline{\gamma}) \subset X, \qquad \underline{\gamma} \in \Gamma_0(\mathfrak{p},\Lambda(\mathfrak{I}))^N,
\end{equation}
such that 
\begin{equation}
\mathbf{\Phi}^{p,q}_\mathfrak{p} (z,\underline{\gamma},\Lambda(\mathfrak{I})) = \int_{\Delta(\underline{\gamma})} E_\mathfrak{p}(z,\psi^{p,q} ,\Lambda(\mathfrak{I})).
\end{equation}
The cocycle property will follow from the fact that for $\gamma_1,\ldots,\gamma_{N+1} \in \Gamma_0(\mathfrak{p},\Lambda(\mathfrak{I}))$ we can find a simplex 
\begin{equation}
\Delta(\gamma_1,\ldots,\gamma_{N+1}) \subset X
\end{equation}
with boundary
\begin{equation}
\partial \Delta(\gamma_1,\ldots,\gamma_{N+1}) = \sum_{1 \leq i \leq N+1} (-1)^{i-1} \Delta(\gamma_1,\ldots,\gamma_{i-1},\gamma_{i+1},\ldots,\gamma_{N+1})
\end{equation}
and such that $E_\mathfrak{p}(z,\psi^{p,q},\Lambda(\mathfrak{I}))$ decreases rapidly on $\Delta(\gamma_1,\ldots,\gamma_{N+1})$ for fixed $z$.

\subsection{Eisenstein series} \label{subsection:Eisenstein_series} For $v \in V$, an $\mathcal{O}$-lattice $\Lambda \subset V_k$ and a holomorphic polynomial $P$ (resp. $Q$) on $V^\vee$ (resp. on $V$), consider the theta series
\begin{equation}
\theta(v,P \otimes \overline{Q}; \psi,\Lambda) := \sum_{\lambda \in \Lambda} \psi(v+\lambda,P \otimes \overline{Q}).
\end{equation}
The series converges rapidly as $\psi(v,P \otimes \overline{Q})$ is rapidly decreasing. By \eqref{eq:psi_invariance_with_coeffs}, we obtain a differential form $\theta(v,P \otimes \overline{Q} ;\psi,\Lambda) \in A^{N-1}(X)$ satisfying
\begin{equation}
\gamma^*\theta(\gamma v,\gamma P \otimes \overline{\gamma Q};\psi,\Lambda) = \theta(v, P \otimes \overline{Q};\psi,\Lambda), \qquad \gamma \in \Gamma(\Lambda) := \mathrm{Aut}_\mathcal{O}(\Lambda) \cap \mathrm{SL}_N(k).
\end{equation}
The Mellin transform of $\theta(v,P \otimes \overline{Q};\psi,\Lambda)$ is the Eisenstein series 
\begin{equation}
\begin{split}
E(v,P \otimes \overline{Q};\psi,\Lambda,s) &:= \int_0^{+\infty} \theta(tv,P \otimes \overline{Q};\psi,t\Lambda) t^{s+N+p-q} \frac{dt}{t} \\
&= \sum_{\lambda \in \Lambda} \eta(v+\lambda,P\otimes \overline{Q},s),
\end{split}
\end{equation}
where $\eta(v,s)$ is given by \eqref{eq:eta_explicit}. Here the sum converges when $\mathrm{Re}(s) \gg 0$ but can be analytically continued to the whole $s$-plane in a standard way using Poisson summation. To do this, consider the scalar product $\langle \cdot, \cdot \rangle$ on $\mathbb{C}^N$ given by 
\begin{equation} \label{eq:inner_product_CN}
\langle v, w \rangle = 2 \mathrm{Re}(v \cdot w^*)
\end{equation}
and define
\begin{equation}
\begin{split}
\Lambda^\vee &= \left\{w \in \mathbb{C}^N \ | \ \langle v, w \rangle \in \mathbb{Z} \quad \forall v \in \Lambda \right\}.
\end{split}
\end{equation}

Given $g \in G$ and a tangent vector $Y \in \mathfrak{p}=T_{eK} X$ we can define a vector $g_*Y \in T_{gK}X$. The invariance property \eqref{eq:psi_invariance_with_coeffs} can be rewritten as
$$
\psi(gv,g_*Y; gP \otimes \overline{g Q}) = \psi(v,Y;P\otimes \overline{Q}).
$$
By \eqref{eq:psi_Four_transform}, $\psi(\cdot, Y; P\otimes \overline{Q}) \in \mathcal{S}(V)$ is an eigenvector for the Fourier transform and so Poisson sumation gives
\begin{equation}
 \begin{split}
 \sum_{\lambda \in \Lambda}  \psi(t(v+\lambda),g_*Y; P \otimes \overline{Q})  &=  \sum_{\lambda \in \Lambda} \psi(tg^{-1}(v+\lambda),Y;g^{-1}(P \otimes \overline{Q})) \\
&=  C \ \mathrm{Vol}(\mathbb{C}^N/\Lambda)^{-1} t^{-2N} \sum_{\lambda \in \Lambda^{\vee}} e^{2\pi i \langle v,\lambda \rangle} \psi(t^{-1}g^*\lambda,Y; g^{-1}(P \otimes \overline{Q})).
\end{split}
\end{equation}
Using this we can write
\begin{multline}
E(v,g_*Y;P \otimes \overline{Q};  \psi,\Lambda,s) \\
\begin{split}
& = \int_0^\infty \theta(tv,g_*Y; P \otimes \overline{Q};\psi,t\Lambda) t^{s+N+p-q} \frac{dt}{t} \\
& = \int_1^\infty  \theta(tv,g_*Y; P \otimes \overline{Q};\psi,t\Lambda) t^{s+N+p-q} \frac{dt}{t} +C \ \mathrm{Vol}(\mathbb{C}^N/\Lambda)^{-1}  
\end{split}
 \\  \cdot \sum_{\lambda \in \Lambda^\vee} e^{2\pi i \langle v,\lambda \rangle}  \int_1^\infty \psi(tg^*\lambda,Y;g^{-1}(P \otimes \overline{Q}))  t^{-s+N+q-p} \frac{dt}{t}.
\end{multline}
The last expression converges for all $s \in \mathbb{C}$ and gives the desired analytic continuation (with no poles since $\psi(0)=\mathcal{F}\psi(0)=0$) of $E(v,P \otimes \overline{Q};\psi,\Lambda,s)$. We set
\begin{equation}
E(v,P \otimes \overline{Q};\psi,\Lambda)=E(v,P \otimes \overline{Q};\psi,\Lambda,0) \in A^{N-1}(X).
\end{equation}

\begin{proposition}
For fixed $v \in \mathbb{C}^N$ and polynomials $P$ and $Q$, the form $E(v,P \otimes \overline{Q};\psi,\Lambda)$ is closed.
\end{proposition}
\begin{proof}
For $t>0$ define the theta series
\[
\theta(tv,P\otimes \overline{Q};\phi,t\Lambda) = \sum_{\lambda \in \Lambda} \phi(t(v+\lambda),P \otimes \overline{Q}),
\]
where $\phi$ is given in \eqref{eq:def_varphi}. The same argument used above shows that
\[
E(v,P \otimes \overline{Q};\phi,\Lambda,s) := \int_0^\infty \theta(tv,P\otimes \overline{Q};\phi,t\Lambda) t^{s+N+p-q} \frac{dt}{t}, \qquad \mathrm{Re} (s) \gg 0,
\]
admits analytic continuation to $s \in \mathbb{C}$ (with no poles). The relation
\begin{equation} \label{eq:d_E_psi_equals_E_varphi}
dE(v,P \otimes \overline{Q};\psi,\Lambda,s) = -\frac{s}{2N} E(v,P\otimes \overline{Q};\phi,\Lambda,s),
\end{equation}
which follows from \eqref{eq:d_eta_Mellin}, proves the claim.
\end{proof}

Thus we may regard $E(\cdot,P \otimes \overline{Q};\psi,\Lambda)$ as a closed differential $(N-1)$-form on $X$ valued on the space of smooth functions $\mathcal{C}^\infty(V)$, and by \eqref{eq:eta_invariance} we have the equivariance property
\begin{equation} \label{eq:Eis_series_invariance_property}
\gamma^*E(\gamma v, \gamma P \otimes \overline{\gamma Q};\psi,\Lambda) = E(v,P \otimes \overline{Q};\psi,\Lambda), \qquad \gamma \in \Gamma(\Lambda).
\end{equation}

For $\Lambda=\Lambda(\mathfrak{I})$ defined in \eqref{eq:def_Lambda_I}, we set
\begin{equation}
E_\mathfrak{p}(v,P \otimes \overline{Q};\psi,\Lambda(\mathfrak{I}),s) = E(v,P \otimes \overline{Q};\psi,\Lambda(\mathfrak{pI}),s)-\mathrm{N}\mathfrak{p} \cdot E(v,P \otimes \overline{Q};\psi,\Lambda(\mathfrak{I}),s)
\end{equation}
and
\begin{equation}
E_\mathfrak{p}(v,P \otimes \overline{Q};\psi,\Lambda(\mathfrak{I})) = E_\mathfrak{p}(v,P \otimes \overline{Q};\psi,\Lambda(\mathfrak{I}),0).
\end{equation}
Again we regard $E_\mathfrak{p}(\cdot,P \otimes \overline{Q};\psi,\Lambda(\mathfrak{I}))$ as a closed differential $(N-1)$-form  on $X$ valued in $\mathcal{C}^\infty(V)$, equivariant under $\Gamma_0(\mathfrak{p},\Lambda(\mathfrak{I}))(=\Gamma(\Lambda(\mathfrak{pI})) \cap \Gamma(\Lambda(\mathfrak{I})) )$.

\subsection{Behaviour on Siegel sets} Fix two coprime ideals $\mathfrak{p}$ and $\mathfrak{I}$ of $\mathcal{O}$ with $\mathfrak{p}$ of prime norm. Recall that proper rational parabolics of $G_k=\mathrm{SL}_N(k)$ are in bijection with proper flags
\begin{equation}
W_\bullet: 0 \subsetneq W_0 \subsetneq \cdots \subsetneq W_r \subsetneq k^N, \quad r \geq 0.
\end{equation}

Before stating our next result we recall the definition of Siegel sets. For a strictly increasing sequence  $\mathbf{J}=\{j_1 < \cdots < j_r\}$ of integers in $\{1,\ldots,N-1\}$, let $W_{j_k}=\langle e_1,\ldots, e_{j_k} \rangle$ and  $P_{\mathbf{J}}$ be the standard parabolic of $\mathrm{SL}_N(k)$ stabilizing the flag
$$
W_\mathbf{J}: 0 \subsetneq W_{j_1} \subsetneq \cdots \subsetneq W_{j_r}  \subsetneq V.
$$
We can write $P_{\mathbf{J}}=NMA$, where (setting $j_0=0$ and $j_{r+1}=N$)
\begin{equation} \label{eq:Langlands_decomposition_standard_parabolic}
\begin{split}
N &= N_\mathbf{J} = \left\{ \begin{pmatrix} 1_{j_1} & * & \cdots & *  \\ 0 & 1_{j_2 - j_1} & \cdots & * \\ 0 & 0 & \ddots & * \\ 0 & 0 & 0 & 1_{j_{r+1} - j_r} \end{pmatrix}\right\} \\
M &= M_\mathbf{J} =   \left\{ \left. \begin{pmatrix} A_1 & 0 & \cdots & 0 \\ 0 & A_2 & \cdots & 0 \\ 0 & 0 & \ddots & 0 \\ 0 & 0 & 0 & A_{r+1} \end{pmatrix} \ \right| \ A_k \in \mathrm{GL}_{j_k - j_{k-1}}(\mathbb{C}), \ |\det(A_k)|=1 \right\} \\ 
A &= A_\mathbf{J} = \left\{ \left. a(t_1,\ldots,t_{r+1}):=\begin{pmatrix} t_1 1_{j_1} & 0 & \cdots & 0 \\ 0 & t_2 1_{j_2 - j_1} & \cdots & 0 \\ 0 & 0 & \ddots & 0 \\ 0 & 0 & 0 & t_{r+1} 1_{j_{r+1} - j_r} \end{pmatrix} \ \right| \ t_k>0, \ \det a(t_1,\ldots,t_{r+1}) = 1 \right\} .
\end{split}
\end{equation}
An element $g \in G$ can be written as
\[
g = n m a k, \quad n \in N, \ m \in M, \ a \in A, \ k \in \mathrm{SU}(N).
\]
In this decomposition $n$ and $a$ are uniquely determined by $g$ and $m$ and $k$ are determined up to an element of $M \cap \mathrm{SU}(N)$. 

For $t \in \mathbb{R}_{>0}$, let
\[
A_t = \{a(t_1,\ldots,t_{r+1}) \in A \ | \ t_k/t_{k+1}  \geq t \text{ for all } k \}.
\]
The Siegel set determined by $t>0$ and a relatively compact set $\omega \subset NM$ is
\begin{equation} \label{eq:std_Siegel_def_1}
S(t,\omega) := \omega A_t \cdot \mathrm{SU}(N) \subset \mathrm{SL}_N(\mathbb{C});
\end{equation}
we refer to its image in $X$ also as a Siegel set.

More generally, suppose that $W_\bullet$ is a proper flag of $k^N$. A {\it Siegel set for the cusp defined by $W_\bullet$} is a set of the form
\[
S(g,t,\omega) := g^{-1} \omega A_t \cdot  \mathrm{SU}(N)
\]
where $g \in \mathrm{SL}_N(k)$ is such that $gW_\bullet$ is a standard flag (i.e. of the form $W_\mathbf{J}$ for some $\mathbf{J}$). 

We say that $W_\bullet$ {\it defines a good cusp} if  $\gamma e_1 \in W_0$ for some $\gamma \in \Gamma_0(\mathfrak{p},\Lambda(\mathfrak{I}))$.

\begin{proposition}\label{thm:good_cusp_rapid_decrease} Suppose that $W_\bullet$ defines a good cusp. If $v \in k^N$ satisfies $(v+\Lambda(\mathfrak{pI})) \cap W_r =\emptyset$, then $E_\mathfrak{p}(v, P \otimes \overline{Q};
\psi,\Lambda(\mathfrak{I}))$ is rapidly decreasing on every Siegel set for the cusp defined by $W_\bullet$.
\end{proposition}
For the proof it will be convenient to work with adeles. 
Given a finite Schwartz function $\phi_f \in \mathcal{S}(V_k(\mathbb{A}_f))$ and $t >0$, let 
\begin{equation}
\theta(v,t,P \otimes \overline{Q}; \phi_f\otimes \psi) = \sum_{\lambda \in k^N} \phi_f(\lambda) \psi(t(v+\lambda),P \otimes \overline{Q})
\end{equation}
and 
\begin{equation}
\begin{split}
E(v,P \otimes \overline{Q};\phi_f \otimes \psi,s) &= \int_0^\infty \theta(v,t,P \otimes \overline{Q};\phi_f\otimes \psi)t^{s+N+p-q}\frac{dt}{t} \\
&= \sum_{\lambda \in k^N} \phi_f(\lambda) \eta(v+\lambda,P\otimes \overline{Q},s).
\end{split}
\end{equation}
Using Poisson summation as in \autoref{subsection:Eisenstein_series}, for $Y \in \wedge^{N-1} \mathfrak{p}$ we may write  
\begin{multline}
\theta(v,g_*Y,t,  P \otimes \overline{Q};\phi_f\otimes \psi) \\
= C \ \mathrm{Vol}(\mathbb{C}^N/\Lambda(\mathfrak{I}))^{-1} t^{-2N} \sum_{\lambda \in V_k} \widehat{\phi_f}(\lambda) e^{2\pi i \langle v, \lambda \rangle} \psi(t^{-1}g^*\lambda,Y;g^{-1}(P \otimes \overline{Q}))
\end{multline}
and 
\begin{multline} \label{eq:Eisenstein_Poisson}
E(v,g_*Y,P \otimes \overline{Q}; \phi_f \otimes  \psi,s) \\ 
\begin{split}
& = \sum_{\lambda \in V_k} \phi_f(\lambda) \int_1^\infty \psi(tg^{-1}(v+\lambda),Y;g^{-1}(P \otimes \overline{Q})) t^{s+N+p-q}\frac{dt}{t} \\ 
& \quad +C \ \mathrm{Vol}(\mathbb{C}^N/\Lambda(\mathfrak{I}))^{-1}   \sum_{\lambda \in V_k} \widehat{\phi_f}(\lambda) e^{2\pi i \langle v,\lambda \rangle}  \int_1^\infty \psi(tg^*\lambda,Y;g^{-1}(P \otimes \overline{Q}))  t^{-s+N+q-p} \frac{dt}{t},
\end{split}
\end{multline}
showing that $E(v,P \otimes \overline{Q};\phi_f \otimes \psi,s)$ admits analytic continuation to $s \in \mathbb{C}$ that is regular at $s=0$.
Note that we can write
\begin{equation}
E_\mathfrak{p}(v,P \otimes \overline{Q};\psi,\Lambda(\mathfrak{I})) = E(v,P \otimes \overline{Q};\phi_f(\mathfrak{p},\mathfrak{I}) \otimes \psi,s)|_{s=0}
\end{equation}
where $\phi_f(\mathfrak{p},\mathfrak{I}) \in \mathcal{S}(V_k(\mathbb{A}_f))$ is given by
\begin{equation}
\phi_f(\lambda;\mathfrak{p},\mathfrak{I}) = \left\{ \begin{array}{cc} 0, & \text{ if } \lambda \notin \Lambda(\mathfrak{pI}) \otimes_\mathcal{O} \widehat{\mathcal{O}}, \\ 1, & \text{ if } \lambda \in \Lambda(\mathfrak{pI}) \otimes_\mathcal{O} \widehat{\mathcal{O}} \text{ and } \lambda_1 \notin \mathfrak{I}^{-1} \otimes_\mathcal{O} \widehat{\mathcal{O}} \\
1-\mathrm{N}(\mathfrak{p}), & \text{ if } \lambda \in \Lambda(\mathfrak{pI}) \otimes_\mathcal{O} \widehat{\mathcal{O}} \text{ and } \lambda_1 \in \mathfrak{I}^{-1} \otimes_\mathcal{O} \widehat{\mathcal{O}}. \end{array} \right.
\end{equation}

\begin{proof}[Proof of \autoref{thm:good_cusp_rapid_decrease}] 
Fix $Y \in \wedge^{N-1} \mathfrak{p}$, and a vector $v \in k^N$ and define $\tilde{\phi_f} \in \mathcal{S}(V_k(\mathbb{A}_f))$ by $\tilde{\phi_f}(\lambda)=\phi_f(\lambda-v;\mathfrak{p},\mathfrak{I})$. For $g=(g_f,g_\infty) \in \mathrm{SL}_N(\mathbb{A}_k)$, polynomials $P$ and $\overline{Q}$ and $t>0$, define
$$
\theta_{P \otimes \overline{Q}}(g,t) = \sum_{\lambda \in k^N} \tilde{\phi_f}(g_f^{-1}\lambda) \psi(tg_\infty^{-1}\lambda,Y;P \otimes \overline{Q})
$$
and
\begin{equation*}
\begin{split}
E_{P \otimes \overline{Q}}(g,s) &= \int_0^\infty \theta_{P \otimes \overline{Q}}(g,t) t^{s+N+p-q} \frac{dt}{t} \\
&= \sum_{\lambda \in k^N} \tilde{\phi_f}(g_f^{-1}\lambda) \eta(tg_\infty^{-1} \lambda, Y; P \otimes \overline{Q},s).
\end{split}
\end{equation*}
Then our Eisenstein series $E_\mathfrak{p}(v,P \otimes \overline{Q};\psi, \Lambda(\mathfrak{I}))$ satisfies
$$
E_\mathfrak{p}(v,(g_\infty)_*Y; P \otimes \overline{Q};\psi, \Lambda(\mathfrak{I}),s) = E_{g_\infty^{-1}(P \otimes \overline{Q})}((g_f=1,g_\infty),s).
$$
Since $E_{P \otimes \overline{Q}}$ is linear in $P$ and $Q$ and $V^{p,q}$ is a polynomial representation of $G$, the expression $g_\infty^{-1}(P \otimes \overline{Q})$ grows at most polynomially on any Siegel set. It follows that it suffices to show that $E_{P \otimes \overline{Q}}((g_f=1,g_\infty),s)$ is rapidly decreasing on every Siegel set for the cusp defined by $W_\bullet$ for all $P \otimes \overline{Q} \in V^{p,q}$; since $E_{P \otimes \overline{Q}}(g,s)$ is an automorphic form, we can check this by showing that the constant term
$$
E_{P \otimes \overline{Q}}((g_f=1,g_\infty),s)_\mathbf{N} = \int_{\mathbf{N}(\mathbb{Q}) \backslash \mathbf{N}(\mathbb{A})} E_{P \otimes \overline{Q}}(n(g_f=1,g_\infty),s) dn
$$
vanishes, where $\mathbf{N}$ denotes the unipotent radical of the parabolic $\mathbf{P}$ corresponding to $W_\bullet$. Let us fix $P$ and $Q$ and drop $P \otimes \overline{Q}$ from the notation and write simply $E(g,s)$ and $E(g,s)_\mathbf{N}$. By transitivity of constant terms, we may assume that $\mathbf{P}$ is maximal, i.e. that the flag $W_\bullet$ consists of just one proper subspace $W_0$ of $k^N$. Note that under our assumptions on $v$ we have $\tilde{\phi_f}(\lambda)=0$ for $\lambda \in W_0$. For $\lambda \in k^N - W_0$, the orbit of $\lambda$ under $\mathbf{N}(\mathbb{Q})$ is $\lambda+W_0$. Writing $\mathbf{N}_\lambda(\mathbb{Q})$ for the stabilizer of $\lambda$ in $\mathbf{N}(\mathbb{Q})$,  we have
\begin{equation*}
\begin{split}
E & ((g_f=1,g_\infty),  s)_\mathbf{N} \\
&= \int_{\mathbf{N}(\mathbb{Q}) \backslash \mathbf{N}(\mathbb{A})} \left( \sum_{\lambda \in k^N - W_0} \tilde{\phi_f}(n_f^{-1}\lambda) \int_0^\infty \psi(t g_\infty^{-1}n_\infty^{-1}\lambda,Y; P \otimes \overline{Q}) t^{s+N+p-q} \frac{dt}{t} \right) dn \\
&= \sum_{\substack{\lambda \in k^N/W_0 \\ \lambda \neq 0}} \int_{\mathbf{N}(\mathbb{Q}) \backslash \mathbf{N}(\mathbb{A})} \left( \sum_{n' \in \mathbf{N}_\lambda(\mathbb{Q}) \backslash \mathbf{N}(\mathbb{Q})} \tilde{\phi_f}((n'n_f)^{-1}\lambda) \eta(tg_\infty^{-1} (n' n_\infty)^{-1} \lambda, Y; P \otimes \overline{Q},s) \right) dn \\
&= \mathrm{Vol}(\mathbf{N}_\lambda(\mathbb{Q}) \backslash \mathbf{N}_\lambda(\mathbb{A}_f)) \\
& \quad \cdot \sum_{\substack{\lambda \in k^N/W_0 \\ \lambda \neq 0}} \int_{\mathbf{N}_\lambda(\mathbb{A}_f) \backslash \mathbf{N}(\mathbb{A}_f)} \tilde{\phi_f}(n_f^{-1}\lambda) dn_f \int_{\mathbf{N}_\lambda(\mathbb{R}) \backslash \mathbf{N}(\mathbb{R})} \eta(tg_\infty^{-1} n_\infty^{-1} \lambda, Y; P \otimes \overline{Q},s) dn_\infty .
\end{split}
\end{equation*}
Let $\gamma \in \Gamma_0(\mathfrak{p},\Lambda(\mathfrak{I}))$ such that $l:=\langle \gamma e_1 \rangle \subseteq W_0$. Since the Schwartz function $\phi_\mathfrak{p}(\mathfrak{I})$, the $\mathfrak{p}$-component of $\phi_{f}$, satisfies
 \begin{equation} \label{eq:condition_2e_1}
 \int_{k_\mathfrak{p}} \phi_\mathfrak{p}(w+xe_1;\mathfrak{I}) dx = 0, \qquad w \in V_k(k_\mathfrak{p}),
 \end{equation}
using that $\phi_\mathfrak{p}(\mathfrak{I})$ is invariant under $\Gamma_0(\mathfrak{p},\Lambda(\mathfrak{I}))$, we compute
\begin{equation*}
\begin{split}
\int_{\mathbf{N}_\lambda(\mathbb{A}_f) \backslash \mathbf{N}(\mathbb{A}_f)} \tilde{\phi_f}(n_f^{-1}\lambda) dn_\mathfrak{p} &= \int_{W_0(\mathbb{A}_f)} \tilde{\phi_f}(\lambda +w) dw \\
&= \int_{W_0(\mathbb{A}_f)} \phi_f(-v+\lambda +w) dw \\
&= \int_{W_0(\mathbb{A}_f)/l(\mathbb{A}_f)} \int_{\mathbb{A}_f} \phi_f(-v+\lambda +w'+x\gamma e_1)  dx dw' \\
&= 0,
\end{split}
\end{equation*}
showing that indeed the constant term is zero.
\end{proof}

\subsection{Tits compactification and modular symbols} 
First recall that the Tits building $\Delta_\mathbb{Q}(\mathbf{G})$ is a simplicial set whose non-degenerate simplices are in bijection with (proper) rational parabolic subgroups $\mathbf{P}$ of $\mathbf{G}$, or equivalently with proper $k$-rational flags
$$
W_\bullet : 0 \subsetneq W_0 \subsetneq \cdots \subsetneq W_r \subsetneq k^N, \quad r \geq 0. 
$$
The stabilizer $\mathbf{P}(W_\bullet)$ of this flag is a rational parabolic of $\mathbf{G}$ that defines an $r$-simplex in $\Delta_\mathbb{Q}(\mathbf{G})$. Its $i$-th face is the simplex corresponding to the flag obtained from $W_\bullet$ by deleting $W_i$ (degenerate simplices correspond to proper flags where we allow $W_i=W_{i+1}$ for any $i$).

For a parabolic subgroup $\mathbf{P}$, denote by $\mathbf{N}_\mathbf{P}$ its unipotent radical and write $\mathbf{L}_\mathbf{P} = \mathbf{P}/\mathbf{N}_\mathbf{P}$ for its Levi quotient, $\mathbf{S}_\mathbf{P}$ for the maximal $\mathbb{Q}$-split torus in the center of $\mathbf{L}_\mathbf{P}$ and $A_\mathbf{P}=\mathbf{S}_\mathbf{P}(\mathbb{R})^0$ for the identity component of the real points of $\mathbf{S}_\mathbf{P}$. Writing $X(\mathbf{L}_\mathbf{P})_\mathbb{Q}$ for the group of rational characters of $\mathbf{L}_\mathbf{P}$, we define $\mathbf{M}_\mathbf{P} = \cap_{\alpha \in X(\mathbf{L}_\mathbf{P})_\mathbb{Q}} \ker \alpha^2$. Then we have the direct product decomposition
$$
\mathbf{L}_\mathbf{P}(\mathbb{R})=\mathbf{M}_\mathbf{P}(\mathbb{R})A_\mathbf{P}.
$$

The simplex in $\Delta_\mathbb{Q}(\mathbf{G})$ corresponding to $\mathbf{P}$ admits a natural geometric realization. To define it, let $\mathfrak{a}_\mathbf{P}$ and $\mathfrak{n}_\mathbf{P}$ be the Lie algebras of $A_\mathbf{P}$ and $N_\mathbf{P}$ respectively, and let $\Phi^+(P,A_\mathbf{P})$ be the set of roots for the adjoint action of $\mathfrak{a}_\mathbf{P}$ on $\mathfrak{n}_{\mathbf{P}}$. 
These roots define a positive chamber
$$
\mathfrak{a}_\mathbf{P}^+ = \left\{ H \in \mathfrak{a}_{\mathbf{P}} \ | \ \alpha(H)>0, \quad  \alpha \in \Phi^+(P,A_\mathbf{P}) \right\}.
$$
Writing $\langle \cdot, \cdot \rangle$ for the Killing form on $\mathfrak{g}$, we define an open simplex
$$
\mathfrak{a}_\mathbf{P}^+(\infty) = \left\{ H \in \mathfrak{a}_{\mathbf{P}}^+ \ | \ \langle H, H \rangle = 1  \right\} \subset \mathfrak{a}_\mathbf{P}^+
$$
and a closed simplex
$$
\overline{\mathfrak{a}_\mathbf{P}^+}(\infty) = \left\{ H \in \mathfrak{a}_{\mathbf{P}} \ | \ \alpha(H) \geq 0, \ \langle H, H \rangle = 1, \quad  \alpha \in \Phi^+(P,A_\mathbf{P}) \right\}
$$
in $\mathfrak{a}_\mathbf{P}$. Note that for $\mathbf{P}$ maximal the Lie algebra $\mathfrak{a}_\mathbf{P}$ is one-dimensional and so $\overline{\mathfrak{a}_\mathbf{P}^+}(\infty)$ is just a point. Moreover, if $\mathbf{Q}$ is another rational parabolic, then $\overline{\mathfrak{a}_\mathbf{Q}^+}(\infty)$ is a face of $\overline{\mathfrak{a}_\mathbf{P}^+}(\infty)$ if and only if $\mathbf{P} \subseteq \mathbf{Q}$. It follows that $\overline{\mathfrak{a}_\mathbf{P}^+}(\infty)$ gives a geometric realization of the simplex in $\Delta_\mathbb{Q}(\mathbf{G})$ corresponding to $\mathbf{P}$, and so the Tits building $\Delta_\mathbb{Q}(\mathbf{G})$ admits the geometric realization
\begin{equation} \label{eq:Tits_building_geometric_realization}
| \Delta_\mathbb{Q}(\mathbf{G}) | \sim \coprod_{\mathbf{P}} \overline{\mathfrak{a}_\mathbf{P}^+}(\infty)/\sim,
\end{equation}
where the union runs over all proper rational parabolics $\mathbf{P}$ of $\mathbf{G}$ and $\sim$ is the equivalence relation induced by the identification of $\overline{\mathfrak{a}_\mathbf{Q}^+}(\infty)$ with a face of $\overline{\mathfrak{a}_\mathbf{P}^+}(\infty)$ whenever $\mathbf{P} \subseteq \mathbf{Q}$. As a set we may write
$$
| \Delta_\mathbb{Q}(\mathbf{G}) | = \coprod_\mathbf{P}  \mathfrak{a}_\mathbf{P}^+(\infty)
$$
as a disjoint union of open simplexes $\mathfrak{a}_\mathbf{P}^+(\infty)$.

\subsubsection{Tits compactification} \label{subsubsection:Tits_compactification} Here we follow \cite{JiMacPherson02} and  \cite[\S  III.12]{BorelJi}. The Tits compactification $_{\mathbb{Q}}\overline{X}^T$ has boundary $|\Delta_\mathbb{Q}(\mathbf{G})|$: as a set we have
$$
_{\mathbb{Q}}\overline{X}^T = X \cup \coprod_\mathbf{P} \mathfrak{a}_{\mathbf{P}}^+(\infty).
$$
The topology on ${_\mathbb{Q}}\overline{X}^T$ can be described in terms of convergent sequences (for a full description see \cite{BorelJi}). Note that we have fixed $x_0 \in X$ corresponding to the maximal compact subgroup $K = \mathrm{SU}(N) \subset G =\mathbf{G}(\mathbb{R})$ and hence a unique Cartan involution $\theta$ of $G$ that fixes $K$ and extends to ${\bf G}$ (namely, $\theta(g)={^t}\overline{g}^{-1}$). There is a unique section $i_0:\mathbf{L}_\mathbf{P} \to \mathbf{P}$ of the quotient map $\mathbf{P} \to \mathbf{L}_\mathbf{P}$ with image invariant under $\theta$. We write $P=\mathbf{P}(\mathbb{R})$, $N_\mathbf{P}=\mathbf{N}_\mathbf{P}(\mathbb{R})$, $A_\mathbf{P}(x_0)=i_0(A_\mathbf{P})$ and $M_\mathbf{P}(x_0)=i_0(\mathbf{M}_\mathbf{P}(\mathbb{R}))$ and obtain the Langlands decomposition (explicitly given by \eqref{eq:Langlands_decomposition_standard_parabolic} for standard parabolics)
$$
P=N_\mathbf{P}A_\mathbf{P}(x_0)\mathbf{M}_\mathbf{P}(x_0).
$$
Writing $X_\mathbf{P}=\mathbf{M}_\mathbf{P}(x_0)/(K \cap \mathbf{M}_\mathbf{P}(x_0))$, this induces a diffeomorphism
\begin{equation} \label{eq:horosph_decomposition}
N_\mathbf{P} \times A_\mathbf{P}(x_0) \times X_\mathbf{P} \to X, \quad (n,a,mK) \mapsto namK.
\end{equation}
The topology  on $_\mathbb{Q} \overline{X}^T$ is characterized by the following properties:
\begin{enumerate}
\item The subspace topology on the boundary $|\Delta_\mathbb{Q}(\mathbf{G})|$ is the quotient topology given by \eqref{eq:Tits_building_geometric_realization}.

\item Let $x \in X$. A sequence $x_n \in {_\mathbb{Q}\overline{X}}^T$, $n \geq 1$, converges to $x$ if and only if $x_n \in X$ for $n \gg 0$ and $x_n$ converges to $x$ in the usual topology of $X$.

\item Let $H_\infty \in \mathfrak{a}_\mathbf{P}^+(\infty)$ and let $(x_j)_{j \geq 1}$ be a sequence in $X$. Write $x_j=n_j \exp(H_j)m_j$ for unique $n_j \in N_\mathbf{P}$, $H_j \in \mathfrak{a}_\mathbf{P}$ and $m_j \in X_\mathbf{P}$ according to the horospherical decomposition \eqref{eq:horosph_decomposition}. Then $x_j \to H_\infty$ if and only if $x_j$ is unbounded and
\begin{enumerate}
\item[(i)] $H_j/||H_j|| \to H_\infty$ in $\mathfrak{a}_\mathbf{P}$,
\item[(ii)] $d(n_j m_j x_0,x_0)/||H_j|| \to 0$,
\end{enumerate}
where $d$ denotes the Riemannian distance on $X$.
\end{enumerate}

With this topology, ${_\mathbb{Q} \overline{X}}^T$ is a Hausdorff space on which $\mathbf{G}(\mathbb{Q})$ acts continuously.

Given points $x \in X$ and $x' \in {_\mathbb{Q}}\overline{X}^T$, we denote by $[x,x']$ the unique oriented geodesic segment starting at $x$ and ending at $x'$. More explicitly, if $x' \in X$, we define $[x,x']$ to be the image of
\begin{equation} \label{eq:s_definition_1}
s(x,x'):[0,1] \to  {_\mathbb{Q}}\overline{X}^T; \quad t \mapsto s (t; x,x'),
\end{equation}
the constant speed parametrization by the unit interval of the unique oriented geodesic segment with $s(0;x,x')=x$ and $s(1;x,x')=x'$. If $x'$ belongs to the boundary of ${_\mathbb{Q}}\overline{X}^T$, then there exists a unique parabolic subgroup $\mathbf{P}$ such that $x'$ corresponds to $H_\infty \in \mathfrak{a}_\mathbf{P}^+(\infty)$. In the coordinates given by \eqref{eq:horosph_decomposition}, we  have $x=n \exp(H) m$, and we define $[x,x']$ to be the image of

\begin{equation} \label{eq:s_definition_2}
s(x,x') : [0,1]  \to {_\mathbb{Q}}\overline{X}^T; \quad t \mapsto s(t;x,x') =  \left\{ \begin{array}{ll} n \exp (H + \frac{t}{1-t}H_\infty ) m, & \mbox{if } t < 1, \\ x', & \mbox{if } t=1. \end{array} \right.
\end{equation}
Given subsets $S \subset X$ and $S' \subset {_\mathbb{Q}}\overline{X}^T$, the cone $C(S,S')$ (also known as the join $S*S'$) is the subset of ${_\mathbb{Q}}\overline{X}^T$ defined as
$$
C(S,S') = \bigcup_{\substack{x \in S \\ x' \in S'}} [x,x'].
$$
If $S=\{x\}$, we say that $C(S,S')$ is the cone on $S'$ with vertex $x$. When $S'$ is given by a simplicial map $\Delta_r \to \Delta_\mathbb{Q}(\mathbf{G})$ into the Tits boundary, the cone on $S'$ with vertex $x$ is naturally the image of an $(r+1)$-simplex $|\Delta_{r+1}| \to {_\mathbb{Q}}\overline{X}^T$ (oriented so that the boundary orientation agrees with that of $S'$). More generally, if $S$ is given by a simplicial map $|\Delta_{k}| \to X$ and $S'$ is given by a simplicial map $\Delta_r \to \Delta_\mathbb{Q}(\mathbf{G})$, then the cone $C(S,S')$  is the image of a map
\begin{equation} 
|\Delta_k| \times |\Delta_r| \times [0,1] \to _{\mathbb{Q}}\overline{X}^T
\end{equation}
that factors through the join
\begin{equation} \label{eq:cones_as_simplicial_maps}
|\Delta_{k+r+1}| \simeq |\Delta_k * \Delta_r| \to _{\mathbb{Q}}\overline{X}^T.
\end{equation}

\subsubsection{Modular symbols} For $k \geq 0$, let $\Delta_k '$ be the first barycentric subdivision of the standard $k$-simplex. Its vertices are in bijection with the non-empty subsets of $\{0,\ldots,k\}$, and a set of vertices $\{v_0,\ldots,v_r\}$ forms an $r$-simplex if and only if they are linearly ordered, i.e. $v_0 \subseteq \cdots \subseteq v_r$. Denote this simplex by $\Delta_{v_0,\ldots,v_r}$.

For a collection $\underline{\gamma}=(\gamma_0,\ldots,\gamma_{k-1})$ of $k \leq N$ elements of $\Gamma_0(\mathfrak{p},\Lambda(\mathfrak{I}))$, let us define a continuous map
$$
\Delta(\underline{\gamma}):|\Delta_{k-1}'| \to {_\mathbb{Q}}\overline{X}^T.
$$
Assume first that $\langle \gamma_0 e_1,\ldots,\gamma_{k-1} e_1 \rangle \neq k^N$. For each chain $v_0 \subseteq \cdots \subseteq v_r$ defining an $r$-simplex in $\Delta'_{k-1}$, the flag
\begin{equation} \label{eq:proper_flag_1}
0 \subsetneq \langle \gamma_i e_1 \ | \ i \in v_0 \rangle \subseteq \langle \gamma_i e_1 \ | \ i \in v_1 \rangle \subseteq \cdots \subseteq \langle \gamma_i e_1 \ | \ i \in v_r \rangle \subsetneq k^N
\end{equation}
is a proper flag of length $r$. We define $\Delta(\underline{\gamma})(\Delta_{v_0,\ldots,v_r})$ to be the corresponding (possibly degenerate) $r$-simplex in $\Delta_\mathbb{Q}(\mathbf{G})$; we give this simplex the orientation induced by $\Delta(\underline{\gamma})$ by the standard orientation on $\Delta_{k-1}'$. This assignment preserves faces and degeneracies and so defines a simplicial map $\Delta(\underline{\gamma})$.

 Next assume that $k=N$ and the vectors $\gamma_0 e_1,\ldots,\gamma_{N-1} e_1$ are linearly independent. Define 
\begin{equation}
A(\underline{\gamma}) = (\gamma_0 e_1|\cdots|\gamma_{N-1} e_1) \in M_N(\mathcal{O}) \cap \mathrm{GL}_N(k)
\end{equation}
 to be the matrix formed by the first columns of $\gamma_0,\ldots,\gamma_{N-1}$. Fix an $N$-th root $(\det A(\underline{\gamma}))^{-1/N}$ of $\det A(\underline{\gamma})^{-1}$ and let $a(\underline{\gamma})=(\det A(\underline{\gamma}))^{-1/N}A(\underline{\gamma})$; the matrix $a(\underline{\gamma})$ has determinant one and defines a point 
\begin{equation} \label{eq:mod_symbol_barycenter}
x_0(\underline{\gamma})=a(\underline{\gamma})K \in X
\end{equation}
(independent of the choice of $N$-th root above). Suppose that $v_0 \subsetneq \cdots \subsetneq v_r$ is a chain defining a non-degenerate $r$-simplex in $\Delta'_{N-1}$. If $v_r \neq \{0,\ldots,N-1\}$, then we define $\Delta(\underline{\gamma})(\Delta_{v_0,\ldots,v_r})$ to be the $r$-simplex of $\Delta_\mathbb{Q}(\mathbf{G})$ corresponding to the flag \eqref{eq:proper_flag_1}. If $v_r = \{0,\ldots,N-1\}$, then we define
\begin{equation} \label{eq:mod_symbol}
\Delta(\underline{\gamma})(\Delta_{v_0,\ldots,v_r}) = \text{cone on } \Delta(\underline{\gamma})(\Delta_{v_0,\ldots,v_{r-1}}) \text{ with vertex } x_0(\underline{\gamma}).
\end{equation}
These assignments are compatible with face maps and therefore give rise to a well-defined continuous map $\Delta(\underline{\gamma}): \Delta_{N-1} ' \to {_\mathbb{Q}}\overline{X}^T$. By induction on $k$ one shows that
\begin{equation} \label{eq:modular_symbol_equivariance}
\Delta(\gamma'\gamma_0,\ldots,\gamma' \gamma_{k-1}) = \gamma'\Delta(\gamma_0,\ldots,\gamma_{k-1}), \qquad \text{ for } \gamma' \in \Gamma_0(\mathfrak{p},\Lambda(\mathfrak{I})).
\end{equation}

Note that when the vectors $\gamma_i e_1$ are linearly dependent, the image of the map $\Delta(\underline{\gamma})$ is contained in the boundary of ${_\mathbb{Q}}\overline{X}^T$. When they are linearly independent, the intersection
\begin{equation} \label{eq:def_mod_symbol_interior}
\Delta^\circ(\underline{\gamma}) := X \cap \mathrm{Im}(\Delta(\underline{\gamma}))
\end{equation}
 of the image of $\Delta(\underline{\gamma})$ with the interior $X$ of ${_\mathbb{Q}}\overline{X}^T$ is a submanifold of dimension $N-1$, namely
\begin{equation} \label{eq:delta_gamma_explicit_description}
\Delta^\circ(\underline{\gamma}) = \{a(\underline{\gamma})\mathrm{diag}(t_1,\ldots,t_N)K \ | \ t_i \in \mathbb{R}_{>0}, \ \Pi_{i=1}^N t_i = 1\} \subset G/K=X.
\end{equation}
(To see this,
we may assume that $\gamma_i e_1 = e_i$, so that $a(\underline{\gamma})$ is the identity matrix $1_N$. Consider first a non-degenerate simplex $\Delta_{v_0,\ldots,v_r}$ in $\Delta_{N-1}'$ with $v_k=\{1,\ldots,|v_k|\}$. If $|v_r| <N$, then the $r$-simplex $\Delta(1_N)(\Delta_{v_0,\ldots,v_r})$ in the Tits compactification of $X$ corresponds to the standard proper flag
$$
0 \subsetneq \langle e_i \ | \ i \leq |v_0| \rangle \subsetneq \cdots \subsetneq  \langle e_i \ | \ i \leq |v_r| \rangle \neq V_k,
$$
and the subgroup $A_\mathbf{P}$ of the corresponding parabolic $\mathbf{P}$ is
$$
A_\mathbf{P} = \left\{ a(t_0,\ldots,t_{r+1}) := \left. \begin{pmatrix} t_0 \cdot 1_{|v_0|} & & \\  & t_1 \cdot 1_{|v_1|-|v_0|} & & \\ & & \ddots & \\ & & & t_{r+1} 1_{N-|v_r|} \end{pmatrix} \ \right| \ t_i >0, \ \det a(t_0,\ldots,t_{r+1}) = 1 \right\}.
$$
If $|v_r|=N$, then the cone of $\Delta(1_N)(\Delta_{v_0,\ldots,v_{r-1}})$ with vertex $x_0$ is
$$
\{a(t_0,\ldots,t_{r+1})K \ | \ a(t_0,\ldots,t_{r+1}) \in A_\mathbf{P}, t_0 \geq \cdots \geq t_{r+1} \}.
$$
The statement follows since any $r$-simplex $\Delta(1_N)(\Delta_{v_0,\ldots,v_r})$ can be obtained as a translate of a simplex corresponding to a standard flag as above by a Weyl group element.)

\medskip

The  coordinates $a(\underline{\gamma})\mathrm{diag}(t_1,\ldots,t_N)K \mapsto t_i$ identify $\Delta^\circ(\underline{\gamma})$ with the manifold $C$ defined in \eqref{eq:C_manifold}. This isomorphism is orientation preserving.\footnote{Recall that we have defined the orientation of $\Delta^\circ(\underline{\gamma})$ to be induced by the boundary and that the orientation on $C$ is fixed in \S \ref{subsection:Integral_on_max_torus}.} For convenience we define $\Delta^\circ(\underline{\gamma}) =\emptyset$ if $A(\underline{\gamma})$ is not invertible.

Note that \eqref{eq:delta_gamma_explicit_description} implies that $\Delta^\circ(\underline{\gamma})$ admits a finite cover by $\mathrm{SL}_N(k)$-translates of (images in $X$ of) standard Siegel sets of the form \eqref{eq:std_Siegel_def_1}. One may take this cover to consist of one Siegel set for every parabolic $\mathbf{P}$ stabilizing a flag consisting of subspaces of the form $\langle \gamma_i e_1 \ | \ i \in I \rangle$ for $I \subsetneq \{0,\ldots,N-1\}$.

\subsection{Evaluation on modular symbols and the cocycle property} 

We can now relate the Eisenstein series $E_\mathfrak{p}(v,P \otimes \overline{Q};\psi,\Lambda(\mathfrak{I}))$  and the Eisenstein cocycle.

\begin{proposition}\label{prop:Eis_integral_on_modular_symbols} Assume that $v $ does not lie in any $\Lambda(\mathfrak{pI})$-translate of a proper subspace of $V$ of the form $\langle \gamma_i e_1 \ | \ i \in I \rangle$ for $I \subseteq \{0,\ldots,N-1\}$. Then $\mathbf{\Phi}^{p,q}_\mathfrak{p}(\cdot,\underline{\gamma},\Lambda(\mathfrak{I}))$ is defined at $v$ and 
$$
\mathbf{\Phi}^{p,q}_\mathfrak{p}(v ,\underline{\gamma},\Lambda(\mathfrak{I})) (P \otimes \overline{Q} ) = \int_{\Delta^\circ(\underline{\gamma})} E_\mathfrak{p}(v,P \otimes \overline{Q};\psi^{p,q} ,\Lambda(\mathfrak{I})) .$$
\end{proposition}

\begin{proof} Consider the matrix $A(\underline{\gamma})=(\gamma_0 e_1 | \ldots | \gamma_{N-1} e_1)$. If $A(\underline{\gamma})$ is not invertible, then both sides are zero by definition. Now assume that $A(\underline{\gamma})$ is invertible and take $A(\underline{\gamma})^{-1}P$ and $\overline{A(\underline{\gamma})^{-1} Q}$ to be monomial, say $A(\underline{\gamma})^{-1}P(z)=z^I=z_1^{i_1} \cdots z_N^{i_N}$ and $\overline{A(\underline{\gamma})^{-1}Q}(\overline{z})=\overline{z}^J=\overline{z_1}^{j_1} \cdots \overline{z_N}^{j_N}$; it suffices to show that with this choice of $P$ and $\overline{Q}$ we have
$$
\int_{\Delta^\circ(\underline{\gamma})} E_\mathfrak{p}(v,P \otimes \overline{Q};\psi^{p,q},\Lambda(\mathfrak{I})) = D^{I,J}_\mathfrak{p}(v,A(\underline{\gamma}),\Lambda(\mathfrak{I})).
$$
Note that the proof of Proposition \ref{thm:good_cusp_rapid_decrease} shows that, for any $s$, the Eisenstein series $E_\mathfrak{p}(v,
P \otimes \overline{Q};\psi, \Lambda(\mathfrak{I}),s)$ is rapidly decreasing on any Siegel set corresponding to a good cusp; since $\Delta^\circ(\underline{\gamma})$ admits a finite cover by such Siegel sets, it follows that $E_\mathfrak{p}(v,
P \otimes \overline{Q};\psi, \Lambda(\mathfrak{I}),s)$ is integrable over $\Delta^\circ(\underline{\gamma})$. For $\mathrm{Re}(s) \gg 0$, we compute
\begin{multline*}
\int_{\Delta^\circ(\underline{\gamma})} E(v,P \otimes \overline{Q};\psi^{p,q},\Lambda(\mathfrak{I}),s) \\
\begin{split}
&= \sum_{\lambda \in \Lambda(\mathfrak{I})} \int_{\Delta^\circ(\underline{\gamma})} \eta^{p,q}(v+\lambda,P \otimes \overline{Q},s) \\
&= \sum_{\lambda \in \Lambda(\mathfrak{I})} \int_{T/T\cap K} a(\underline{\gamma})^*\eta^{p,q} (v+\lambda,P \otimes \overline{Q},s) \\
&= \sum_{\lambda \in \Lambda(\mathfrak{I})} \int_{T/T\cap K} \eta^{p,q} (a(\underline{\gamma})^{-1}(v+\lambda),a(\underline{\gamma})^{-1}(P \otimes \overline{Q}),s) \\
&= |\det A(\underline{\gamma})|^{-s/N} \det A(\underline{\gamma})^{-1} \sum_{\lambda \in \Lambda(\mathfrak{I})} \int_{T/T\cap K} \eta^{p,q} (A(\underline{\gamma})^{-1}(v+\lambda),A(\underline{\gamma})^{-1}(P \otimes \overline{Q}),s),
\end{split}
\end{multline*}
where the last equality follows from the homogeneity property \eqref{eta_homogeneity}. The desired identity follows by analytic continuation from  \autoref{lemma:eta_integration}.
\end{proof}

We can now use \autoref{prop:Eis_integral_on_modular_symbols} to prove that $\mathbf{\Phi}^{p,q}_\mathfrak{p}(\Lambda(\mathfrak{I}))$ is indeed an $(N-1)$-cocycle, i.e. that it satisfies property \eqref{eq:cocycle_def_closed}.

Given $N+1$ elements $\gamma_0,\ldots,\gamma_N \in \Gamma_0(\mathfrak{p},\Lambda(\mathfrak{I}))$, write $S_j$ for $\Delta(\gamma_0,\ldots,\widehat{\gamma_j}, \ldots, \gamma_N )$ and fix $x \in X$ such that
$$
x \notin 
\bigcup_{0 \leq j \leq N} S_j.
$$
 Then $S_j$ is an (oriented) $(N-1)$-simplex in the boundary of ${_\mathbb{Q}}\overline{X}^T$, and we denote by $(-1)^j S_j$ the same simplex with opposite orientation if $j$ is odd. Note that $\sum_j (-1)^j S_j$ is a cycle, i.e. $\sum (-1)^j \partial S_j=0$. For each $j$ with $0 \leq j \leq N$, we next define an $N$-simplex $C_j:|\Delta_N| \to {_\mathbb{Q}}\overline{X}^T$. Assume first that $\langle \gamma_0 e_1,\ldots,\widehat{\gamma_j e_1}, \ldots, \gamma_N e_1 \rangle \neq k^N$. We define $C_j$ to be the cone on $(-1)^j S_j$ with vertex $x$ (cf. \S  \ref{subsubsection:Tits_compactification}); its boundary is $\partial C_j = (-1)^j S_j - (-1)^j C(x,\partial S_j)$.

Now assume that $\langle \gamma_0 e_1,\ldots,\widehat{\gamma_j e_1}, \ldots, \gamma_N e_1 \rangle = k^N$. Then the intersection $S^\circ_j$ of $S_j$ with $X$ is non-empty, and in \eqref{eq:mod_symbol_barycenter} we have defined a barycenter $x_j:=x_0(\gamma_0,\ldots,\widehat{\gamma_j}, \ldots, \gamma_N) \in S^\circ_j$ such that $S_j$ is the cone on $\partial S_j$ with vertex $x_j$.  We define $C_j$ to be the cone $C([x,x_j], \partial S_j)$, where $[x,x_j]$ denotes the oriented geodesic segment from $x$ to $x_j$. More explicitly, let $s:[0,1] \to [x,x_j]$ be the constant speed parametrization of the geodesic segment joining $s(0)=x$ to $s(1)=x_j$. For each simplex $\Delta(\gamma_0 ,\ldots, \widehat{\gamma_j }, \ldots, \gamma_N )(\Delta_{v_0,\ldots,v_r})$ contained in $\partial S_j$, let $\mathbf{P}$ be the corresponding parabolic; writing $s(t)=n(t)\exp(H_t)m(t)$ we obtain a map
\begin{equation} \label{eq:explicit_cone_cocycle_relation}
[0,1] \times \mathfrak{a}_\mathbf{P}^+ \to X, \qquad (t,H') \mapsto n(t)\exp(H_t+H')m(t),
\end{equation}
whose closure is $C([x,x_j],\Delta(\gamma_0,\ldots, \widehat{\gamma_j}, \ldots, \gamma_N)(\Delta_{v_0,\ldots,v_r}))$. Since $\partial S_j$ has empty boundary, the boundary of $C_j$ is the union of $S_j$ (= the cone on $S_j^\circ$ with vertex $x_j$) and the cone $C(x,\partial S_j)$; we orient $C_j$ so that the induced orientation on $S_j$ is that given by \eqref{eq:mod_symbol}, so that $\partial C_j = (-1)^j S_j - (-1)^j C(x,\partial S_j)$. 

It follows that the sum $\Delta(\gamma_0,\ldots,\gamma_N)=\sum_j C_j$ has boundary $\sum_j (-1)^j S_j$. The cocycle property \eqref{eq:cocycle_def_closed} follows immediately from Stokes' theorem and the following lemma.

\begin{lemma}
Assume that $v $ does not lie in any $\Lambda(\mathfrak{pI})$-translate of a proper subspace of $V$ of the form $\langle \gamma_i e_1 \ | \ i \in I \rangle$ for $I \subseteq \{0,\ldots,N\}$. Then the Eisenstein series $E_\mathfrak{p}(v,P \otimes \overline{Q};\psi^{p,q} ,\Lambda(\mathfrak{I}))$ is rapidly decreasing  on $\Delta(\gamma_0,\ldots,\gamma_N)$.
\end{lemma}
\begin{proof}
By \autoref{thm:good_cusp_rapid_decrease}, it suffices to show that each $C_j$ can be covered by finitely many Siegel sets of good cusps, which is obvious from the explicit description \eqref{eq:explicit_cone_cocycle_relation}.
\end{proof}

\section{Eisenstein cocycle and critical values of Hecke L-series}

\subsection{Units of extensions of $k$} Let $L$ be a field extension of $k$ of degree $N \geq 2$. We denote its ring of integers by $\mathcal{O}_L$ and write $\sigma_1,\ldots,\sigma_N$ for the complex embeddings of $L$ in $\mathbb{C}$ extending $\sigma$. We obtain an embedding
\[
\underline{\sigma} \in \mathrm{Hom}_\mathcal{O}(L, \mathbb{C}^N), \quad \underline{\sigma}(l) = (\sigma_1(l),\ldots,\sigma_N(l)).
\] 
Let $n:L^\times \to k^\times$ be the norm map and $L^1$ be the kernel of $n$. We fix ideals $\mathfrak{a}$, $\mathfrak{P}$ and $\mathfrak{f}$ of $\mathcal{O}_L$ that are pairwise coprime and such that $\mathfrak{p}:=n(\mathfrak{P})$ is prime. We let $U(\mathfrak{f})=\mathcal{O}_L^\times \cap (1+\mathfrak{f})$ and $U(\mathfrak{f})^1=U(\mathfrak{f}) \cap L^1$. We denote by $U(\mathfrak{f})^1_{\mathrm{tors}}$ the torsion subgroup of $U(\mathfrak{f})^1$ and fix units $u_1,\ldots,u_{N-1} \in U(\mathfrak{f})$ that generate a subgroup $U(\mathfrak{f})':=\langle u_1,\ldots,u_{N-1} \rangle$ of $U(\mathfrak{f})^1$ that is free abelian of rank $N-1$ and maps bijectively to $U(\mathfrak{f})^1/U(\mathfrak{f})^1_{\mathrm{tors}}$ via the quotient map.

\begin{lemma} \label{L41}
Let $\mathfrak{I}$ be a fractional ideal of $\mathcal{O}$ coprime to $\mathfrak{p}$ and isomorphic to $(\mathrm{det}_\mathcal{O}(\mathfrak{fa}^{-1}))^{-1}$. There exists a $k$-isomorphism $\alpha: L \xrightarrow{\sim} k^N$ making the diagram
$$
\begin{tikzcd}
\mathfrak{fa}^{-1} \arrow[r,"\alpha"', "\sim"] \arrow[d,hook] & \Lambda(\mathfrak{I}) \arrow[d,hook] \\ \mathfrak{f(aP)}^{-1} \arrow[r,"\alpha"', "\sim"] & \Lambda(\mathfrak{pI})
\end{tikzcd}
$$
commute.
\end{lemma}
\begin{proof}
Fix an isomorphism $\tilde{\alpha}:\mathfrak{fa}^{-1} \xrightarrow{\sim} \Lambda(\mathfrak{I})$. Then the $\mathcal{O}$-lattices $\Lambda_1=\Lambda(\mathfrak{pI})$ and $\Lambda_2=\tilde{\alpha}(\mathfrak{f(aP)}^{-1})$  contain $\Lambda(\mathfrak{I})$ and $\Lambda_i/\Lambda(\mathfrak{I})\simeq \mathcal{O}/\mathfrak{p}$ for $i=1,2$. For each finite place $v$ of $k$ and $i=1,2$ we obtain an $\mathcal{O}_v$-lattice $\Lambda_{i,v}=\Lambda_i \otimes_\mathcal{O} \mathcal{O}_v$ in $k_v^N$, and we have $\Lambda_{1,v}=\Lambda_{2,v}$ for all $v \neq \mathfrak{p}$. Pick $g_\mathfrak{p} \in \mathrm{SL}_N(\Lambda(\mathfrak{I})_\mathfrak{p})$ such that $g_\mathfrak{p}\Lambda_{2,\mathfrak{p}} = \Lambda_{1,\mathfrak{p}}$ and let 
\begin{equation*}
U_\mathfrak{p} = \mathrm{SL}_N(\Lambda(\mathfrak{I})_\mathfrak{p}) \cap \mathrm{SL}_N(\Lambda_{2,\mathfrak{p}}).
\end{equation*}
Then $U_\mathfrak{p}$ is an open compact subgroup of $\mathrm{SL}_N(k_\mathfrak{p})$ and $U=g_\mathfrak{p}U_\mathfrak{p} \times \prod_{v \neq \mathfrak{p}} \mathrm{SL}_N(\Lambda(\mathfrak{I})_v)$ is an open subset of $\mathrm{SL}_N(\mathbb{A}_{k,f})$. Since $\mathrm{SL}_N(k)$ is dense in $\mathrm{SL}_N(\mathbb{A}_{k,f})$, we may find $g \in \mathrm{SL}_N(k) \cap U$. Then $g$ stabilizes $\Lambda(\mathfrak{I})$ and $g\Lambda_2=\Lambda_1$ (since $g\Lambda_{2,v}=\Lambda_{1,v}$ for every finite place $v$), and so $\alpha:=g \circ \tilde{\alpha}$ makes the diagram in the statement commute.

\end{proof}

From now on we fix an isomorphism $\alpha$ as in the above lemma. These choices define:

\begin{itemize}
\item a vector
\begin{equation}
v_0 = \alpha(1) \in k^N;
\end{equation}
\item a $k$-basis $\alpha_j=\alpha^{-1}(e_j)$ ($j=1,\ldots,N$) of $L$ and a matrix 
$$
a_\alpha =(\sigma_i(\alpha_j))^{-1} \det(\sigma_i(\alpha_j))^{1/N} \in \mathrm{SL}_N(\mathbb{C})
$$
(here $\det(\sigma_i(\alpha_j))^{1/N}$ denotes a fixed $N$-th root of $\det(\sigma_i(\alpha_j))$);
\item an inclusion
\begin{equation}
\iota_\alpha:L^\times \rightarrow \mathrm{GL}_N(k)
\end{equation}
sending $l \in L^\times$ to the map $x \mapsto \alpha(l\alpha^{-1}(x))$, that can be described using $a_\alpha$:
\begin{equation}
\iota_\alpha(l)= a_\alpha \mathrm{diag}(\underline{\sigma}(l)) a_\alpha^{-1}.
\end{equation}
Define
$$
\Gamma_1(v_0, \Lambda(\mathfrak{I})) = \left\{ \gamma \in \Gamma(\Lambda(\mathfrak{I})) \ | \ (\gamma-1)v_0 \in \Lambda(\mathfrak{I})  \right\}
$$
and
$$
\Gamma := \Gamma_0(\mathfrak{p},\Lambda(\mathfrak{I})) \cap \Gamma_1(v_0,\Lambda(\mathfrak{I})).
$$
Since multiplication by $u \in U(\mathfrak{f})^1$ induces an $\mathcal{O}$-linear automorphism of $\mathfrak{fa}^{-1}$ and $\mathfrak{f(aP)}^{-1}$ of determinant $1$ that preserves $1+\mathfrak{fa}^{-1}$, the restriction of $\iota_\alpha$ to $U(\mathfrak{f})^1$ defines an inclusion
\begin{equation}
\iota_\alpha:U(\mathfrak{f})^1 \rightarrow \Gamma.
\end{equation}
\item Write $(L \otimes_{k,\sigma} \mathbb{C})^1$ for the elements of $(L \otimes_{k,\sigma} \mathbb{C})^\times$ of norm $1$ and $(L \otimes_{k,\sigma} \mathbb{C})^1_c$ for the maximal compact subgroup of $(L \otimes_{k,\sigma} \mathbb{C})^1$. The map $u \mapsto \iota_\alpha(u)a_\alpha(=a_\alpha \mathrm{diag}(\underline{\sigma}(u)))$ induces an embedding
\begin{equation} \label{eq:def_map_symmetric_spaces_L_to_SL_N}
\iota_\alpha: (L\otimes_{k,\sigma} \mathbb{C})^1 / (L\otimes_{k,\sigma} \mathbb{C})^1_c \to X
\end{equation}
and hence a basepoint $x_\alpha=a_\alpha \mathrm{SU}(N) \in X$ and a map
\begin{equation} \label{eq:def_map_loc_symmetric_spaces_L_to_SL_N}
\iota_\alpha: X(\mathfrak{f}):=U(\mathfrak{f})^1\backslash (L\otimes_{k,\sigma} \mathbb{C})^1 / (L\otimes_{k,\sigma} \mathbb{C})^1_c \to \Gamma \backslash X.
\end{equation}
By Kronecker's theorem (``algebraic integers all of whose conjugates are of norm one are roots of unity''), the kernel of the action of $U(\mathfrak{f})^1$ on $(L\otimes_{k,\sigma} \mathbb{C})^1 / (L\otimes_{k,\sigma} \mathbb{C})^1_c$ equals the torsion subgroup $U(\mathfrak{f})^1_{\mathrm{tors}}$ of $U(\mathfrak{f})^1$, and the action of $U(\mathfrak{f})'=\langle u_1,\ldots,u_{N-1} \rangle \simeq U(\mathfrak{f})^1/U(\mathfrak{f})^1_{\mathrm{tors}}$ on $(L\otimes_{k,\sigma} \mathbb{C})^1 / (L\otimes_{k,\sigma} \mathbb{C})^1_c$ is free.
We fix the orientation on $(L \otimes_{k,\sigma}\mathbb{C})^1/(L \otimes_{k,\sigma} \mathbb{C})^1_c$ associated to the canonical orientation of $\C^N$ and write 
$$
[X(\mathfrak{f})] \in \mathrm{H}_{N-1}(X(\mathfrak{f}),\mathbb{Z}) \simeq \mathrm{H}_{N-1}(\langle u_1,\ldots,u_{N-1}\rangle ,\mathbb{Z})
$$
for the fundamental class of the (compact, oriented) $(N-1)$-manifold $X(\mathfrak{f})$. We write $\mathrm{cor}:\mathrm{H}_*(\langle u_1,\ldots,u_{N-1} \rangle,\mathbb{Q}) \to \mathrm{H}_*(U(\mathfrak{f})^1,\mathbb{Q})$ and $\mathrm{res}: \mathrm{H}_*(U(\mathfrak{f})^1,\mathbb{Q}) \to \mathrm{H}_*(\langle u_1,\ldots,u_{N-1} \rangle,\mathbb{Q})$ for the corestriction and restriction maps respectively and set
$$
Z_{\mathfrak{f}}= [U(\mathfrak{f}):U(\mathfrak{f})']^{-1}  \mathrm{cor}[X(\mathfrak{f})] \in \mathrm{H}_{N-1}(U(\mathfrak{f})^1,\mathbb{Q}).
$$

\item The embeddings $\sigma_1,\ldots,\sigma_N: L \to \mathbb{C}$ give a basis for $(L \otimes_{k,\sigma} \mathbb{C})^\vee$; we denote by $\tfrac{\partial}{\partial \sigma_1},\ldots, \tfrac{\partial}{\partial \sigma_N}$ the dual basis of $(L \otimes_{k,\sigma} \mathbb{C})^{\vee\vee} \simeq L \otimes_{k,\sigma} \mathbb{C}$. Writing $\alpha_\mathbb{C}=\alpha \otimes 1$ for the extension of $\alpha: L \to k^N$ to an isomorphism $L \otimes_{k,\sigma} \mathbb{C} \to V$, we define polynomials
\begin{equation}
\begin{split}
P_\alpha &= \alpha_\mathbb{C}\left( \frac{\partial}{\partial \sigma_1} \right) \cdots \alpha_\mathbb{C}\left( \frac{\partial}{\partial \sigma_N} \right) \in \mathrm{Sym}^N V, \\
\overline{Q_\alpha} &= \overline{n} \circ \alpha_\mathbb{C}^{-1} \in \mathrm{Sym}^N \overline{V}^\vee.
\end{split}
\end{equation}
Note that these polynomials satisfy
\begin{equation} \label{eq:P_and_Q_alpha_property}
a_\alpha^{-1}P_\alpha = \det(\sigma_i(\alpha_j))^{-1} (e_1 \cdots e_N), \qquad \overline{a_{\alpha}^{-1}Q_\alpha} = \overline{\det(\sigma_i(\alpha_j))} \cdot \overline{z_1\cdots z_N}.
\end{equation}
For non-negative integers $p,q$, we define
$$
P_{\alpha}^{p,q} =  p!^{-N}  \cdot P_\alpha^p \otimes \overline{Q_\alpha}^q \in (V^{pN,qN})^\vee.
$$
Then $P_\alpha^{p,q}$ is invariant under $\iota_\alpha(U(\mathfrak{f})^1)$.
We define
$$
Z_\mathfrak{f}^{p,q} = Z_\mathfrak{f} \otimes P_\alpha^{p,q} \in  \mathrm{H}_{N-1}(U(\mathfrak{f})^1,(V^{pN,qN})^\vee).
$$
\end{itemize}

Let $\mathrm{res}(E_\mathfrak{p}(v_0;\psi^{pN,qN}, \Lambda(\mathfrak{I}))) \in \mathrm{H}^{N-1}(U(\mathfrak{f})^1,V^{pN,qN})$  be the cohomology class defined by the restriction of the closed form $E_\mathfrak{p}(v_0;\psi^{pN,qN}, \Lambda(\mathfrak{I}))$ and define
\begin{equation}
\begin{split}
\langle E_\mathfrak{p}(v_0;\psi^{pN,qN},\Lambda(\mathfrak{I})),Z_\mathfrak{f}^{p,q} \rangle &= \mathrm{res}(E_\mathfrak{p}(v_0; \psi^{pN,qN}, \Lambda(\mathfrak{I}))) \cap Z_\mathfrak{f}^{p,q} \\
&=[U(\mathfrak{f}):U(\mathfrak{f})']^{-1}  \left. \int_{X(\mathfrak{f})} \iota_\alpha^*E_\mathfrak{p}(v_0,P_\alpha^{p,q};\psi^{pN,qN},\Lambda(\mathfrak{I}),s) \right|_{s=0}.
\end{split}
\end{equation}

\subsection{Partial zeta functions} Given integers $p,q \geq 0$, define the partial zeta function
\begin{equation}
\zeta^{p,q}_\mathfrak{f}(\mathfrak{a},s) = \sideset{}{'}\sum_{x \in U(\mathfrak{f})\backslash 1+\mathfrak{fa}^{-1}} \frac{\overline{n(x)}^q}{n(x)^{p+1} |n(x)|^{2s}}, \qquad \mathrm{Re}(s) \gg 0.
\end{equation}
(Since $u\overline{u}=1$ for every $u \in \mathcal{O}^\times$, this is well-defined provided that $p+q+1$ is divisible by the order of the subgroup $n(U(\mathfrak{f}))$ of $\mathcal{O}^\times$, which we assume.) Define also the `$\mathfrak{P}$-smoothed' partial zeta function
\begin{equation}
\zeta^{p,q}_{\mathfrak{f},\mathfrak{P}}(\mathfrak{a},s)= \mathrm{N}\mathfrak{P}^{-s} \zeta^{p,q}_\mathfrak{f}(\mathfrak{aP},s)-\mathrm{N}\mathfrak{P}^{1-s} \zeta^{p,q}_\mathfrak{f}(\mathfrak{a},s).
\end{equation}
These partial zeta functions admit meromorphic continuation to $s \in \mathbb{C}$ that is regular at $s=0$.

\begin{proposition} \label{prop:Eisenstein_zeta_value}
$$
\langle E_\mathfrak{p}(v_0; \psi^{pN,qN}, \Lambda(\mathfrak{I})),Z_\mathfrak{f}^{p,q} \rangle = \det(\sigma_i(\alpha_j)) \zeta^{p,q}_{\mathfrak{f},\mathfrak{P}}(\mathfrak{a},0).
$$
\end{proposition}

\begin{proof}
For $s$ in the range of convergence of the Eisenstein series, we compute
\begin{multline}
\int_{X(\mathfrak{f})}  \iota_\alpha^*E(v_0,P_\alpha^{p,q};\psi^{pN,qN},\Lambda(\mathfrak{I}),s) \\
\begin{split}
&= \int_{U(\mathfrak{f})' \backslash (L\otimes_{k,\sigma} \mathbb{C})^1 / (L\otimes_{k,\sigma} \mathbb{C})^1_c } \iota_\alpha^*\left(\sum_{v \in v_0 + \Lambda(\mathfrak{I})} \eta^{pN,qN} (v,P_\alpha^{p,q},s) \right) \\
&= \int_{U(\mathfrak{f})' \backslash (L\otimes_{k,\sigma} \mathbb{C})^1 / (L\otimes_{k,\sigma} \mathbb{C})^1_c } \iota_\alpha^*\left(\sum_{x \in 1+\mathfrak{fa}^{-1}} \eta^{pN,qN} (\alpha(x),P_\alpha^{p,q},s) \right) \\
&= \int_{U(\mathfrak{f})' \backslash (L\otimes_{k,\sigma} \mathbb{C})^1 / (L\otimes_{k,\sigma} \mathbb{C})^1_c } \iota_\alpha^*\left(\sum_{x \in U(\mathfrak{f})' \backslash 1+\mathfrak{fa}^{-1}} \sum_{u \in U(\mathfrak{f})'} \eta^{pN,qN} (\alpha(ux),P_\alpha^{p,q},s) \right) \\
&= \sum_{x \in U(\mathfrak{f})' \backslash 1+\mathfrak{fa}^{-1}} \int_{(L\otimes_{k,\sigma} \mathbb{C})^1 / (L\otimes_{k,\sigma} \mathbb{C})^1_c } \iota_\alpha^* \eta^{pN,qN} (\alpha(x),P_\alpha^{p,q},s).
\end{split}
\end{multline}
Writing $T$ for the torus of diagonal matrices in $G$, note that the image of $\iota_\alpha$ is identified with the translate $a_\alpha(T/T\cap K) \subset X$. Since $(a_\alpha^{-1}v)_i = \det(\sigma_i(\alpha_j))^{-1/N} \sigma_i(\alpha^{-1}(v))$, using \eqref{eq:P_and_Q_alpha_property} and \autoref{lemma:eta_integration} and writing $\Delta = \det(\sigma_i(\alpha_j))^{-1/N}$ we compute
\begin{multline*} 
\int_{(L\otimes_{k,\sigma} \mathbb{C})^1 / (L\otimes_{k,\sigma} \mathbb{C})^1_c }  \iota_\alpha^* \eta^{pN,qN} (\alpha(x),P_\alpha^{p,q},s) \\ \begin{split}
& = \int_{T/T \cap K} \eta^{pN,qN} (a_\alpha^{-1}\alpha(x),a_\alpha^{-1}(P_\alpha^{p,q}),s) \\
& = p!^{-N} \int_{T/T\cap K} \eta^{pN,qN} (\Delta \underline{\sigma}(x),\Delta^{pN} (e_1 \cdots e_N)^p \otimes \overline{\Delta}^{-qN} \overline{z_1\cdots z_N}^q,s) \\
& = p!^{-N} \Delta^{pN} \overline{\Delta}^{-qN} \int_{T/T\cap K} \eta^{pN,qN} (\Delta \underline{\sigma}(x),(e_1 \cdots e_N)^p \otimes \overline{z_1\cdots z_N}^q,s) \\
& = \Delta^{pN} \overline{\Delta}^{-qN} p!^{-N} \Gamma(\frac{s}{2N}+1+p)^N \prod_{k=1}^N \frac{(\overline{\Delta \sigma_k(x)})^q}{|\Delta \sigma_k(x)|^{s/N} (\Delta \sigma_k(x))^{p+1}} \\
& = |\Delta|^{-s} \Delta^{-N} p!^{-N} \Gamma(\frac{s}{2N}+1+p)^N \frac{\overline{n(x)^q}}{|n(x)|^{s/N} n(x)^{p+1}}
\end{split}
\end{multline*}
and the statement follows.
\end{proof}

\subsection{Moving cycles to the Tits boundary} We now give a fundamental domain $\mathcal{D}$ for the action of $\langle u_1, \ldots, u_{N-1} \rangle$ on the image of the map $\iota_\alpha$ defined in \eqref{eq:def_map_symmetric_spaces_L_to_SL_N}, and a decomposition of $\mathcal{D}$ into $(N-1)$-simplices indexed by the symmetric group $S_{N-1}$.

\subsubsection{Simplices} Let us first define the relevant simplices. For $k \geq 0$, $x \in X$ and $\underline{\gamma}=(\gamma_0,\ldots,\gamma_k) \in \Gamma_0(\mathfrak{p},\Lambda(\mathfrak{I}))^{k+1}$, we define a continuous map
$$
\tilde{\Delta}(\underline{\gamma},x) : |\Delta_k| \to X
$$
inductively on $k$ as follows:
\begin{itemize}
\item[-] For $k=0$ we have $\Delta_k=\{*\}$ and we set $\tilde{\Delta}(\gamma_0,x)(*)=\gamma_0 x$.

\item[-] Assume that $k \geq 1$ and that we have defined $\tilde{\Delta}(\gamma_0,\ldots,\gamma_{k-1},x)$ for every collection of elements $\gamma_0,\ldots,\gamma_{k-1} \in \Gamma_0(\mathfrak{p},\Lambda(\mathfrak{I}))$. We define $\tilde{\Delta}(\gamma_0,\ldots,\gamma_k,x)$ to be the cone on $\tilde{\Delta}(\gamma_0,\ldots,\gamma_{k-1},x)$ with vertex $\gamma_k x$ (we orient this cone by declaring that its vertices $\gamma_0 x, \ldots, \gamma_k x$ are in increasing order).
\end{itemize}

By induction on $k$ one shows that
\begin{equation} \label{eq:simplex_equivariance_1}
\tilde{\Delta}(\gamma'\gamma_0,\ldots,\gamma' \gamma_{k-1},x) = \gamma'\tilde{\Delta}(\gamma_0,\ldots,\gamma_{k-1},x), \qquad \text{ for } \gamma' \in \Gamma_0(\mathfrak{p},\Lambda(\mathfrak{I})).
\end{equation}

\subsubsection{Fundamental domain} Consider now the fundamental domain $\mathcal{D}$ for the action of $\langle u_1, \ldots, u_{N-1} \rangle$ on the image of the map $\iota_\alpha$ defined as follows: for $\underline{t} \in [0,1]^{N-1}$, let 
$$
\sigma(\underline{u})(\underline{t}) = (\sigma_1(u_1)^{t_1}\cdots \sigma_1(u_{N-1})^{t_{N-1}}, \ldots, \sigma_N(u_1)^{t_1}\cdots \sigma_N(u_{N-1})^{t_{N-1}}) \in (\mathbb{C}^\times)^{N}
$$
and let
$$
\mathcal{D} = \{ a_\alpha \sigma(\underline{u})(\underline{t})K \ | \ \underline{t} \in [0,1]^{N-1} \} \subset X.
$$
There is a standard decomposition of $[0,1]^{N-1}$ into $(N-1)$-simplices:
$$
[0,1]^{N-1} = \bigcup_{\sigma \in S_{N-1}} \{ (t_1,\ldots,t_{N-1})  \in [0,1]^{N-1} \ | \ t_{\sigma(1)} \leq \cdots \leq t_{\sigma(N-1)} \}.
$$
This induces a corresponding simplicial decomposition of $\mathcal{D}$: writing $U_i=\iota_\alpha(u_i) \in \Gamma_0(\mathfrak{p},\Lambda(\mathfrak{I}))$ and 
\begin{equation} \label{eq:def_u_sigma}
\underline{u}_\sigma = (1,U_{\sigma(1)},U_{\sigma(1)}U_{\sigma(2)},\ldots,U_{\sigma(1)} \cdots U_{\sigma(N-1)}),
\end{equation}
we have
\begin{equation} \label{eq:fund_domain_decomposition}
\mathcal{D} = \sum_{\sigma \in S_{N-1}} \mathrm{sgn}(\sigma) \tilde{\Delta}(\underline{u}_\sigma,x_\alpha).
\end{equation}

\subsubsection{Deforming $\tilde{\Delta}(\underline{\gamma},x)$} 
Let $k \geq 0$, $x \in X$ and $\underline{\gamma} \in \Gamma_0(\mathfrak{p},\Lambda(\mathfrak{I}))^{k+1}$ with $k < N$. In this paragraph we define a homotopy between the simplices $\tilde{\Delta}(\underline{\gamma},x)$ and $\Delta(\underline{\gamma})$; that is, a map
$$
H(\underline{\gamma},x) : | \Delta_{k} | \times [0,1] \to {_\mathbb{Q}}\overline{X}^T
$$
such that
\begin{equation} \label{eq:H_homotopy_property}
\begin{split}
H(\underline{\gamma},x)|_{|\Delta_{k} | \times \{0\}} &= \tilde{\Delta}(\underline{\gamma},x) \\
H(\underline{\gamma},x)|_{|\Delta_{k} | \times \{1\}} &= \Delta(\underline{\gamma})
\end{split}
\end{equation}
Since the cones in $X$ depend on the order of the vertices we try to define this homotopy with a bit of care. 

Moreover, we will show that $H$ can be covered by a finite number of Siegel sets attached to good cusps (recall that we say that a cusp corresponding to a rational flag $W_\bullet$ is good if we can find $\gamma \in \Gamma_0(\mathfrak{p},\Lambda(\mathfrak{I}))$ such that $\gamma e_1 \in W_0$) and has the equivariance property
\begin{equation} \label{eq:homotopy_equivariance_1}
H(\gamma'\gamma_0,\ldots,\gamma' \gamma_k,x) = \gamma' H(\gamma_0,\ldots,\gamma_k,x), \qquad \text{ for } \gamma' \in \Gamma_0(\mathfrak{p},\Lambda(\mathfrak{I})).
\end{equation}

To define $H$ we use a decomposition of $|\Delta_k| \times [0,1]$ defined inductively as follows. For $k=0$ we take the decomposition of $\{*\} \times [0,1] \simeq [0,1]$ with $0$-simplices $\{0\}$ and $\{1\}$ and the $1$-simplex $(0,1)$. The decomposition of $|\Delta_{k}| \times [0,1]$ is defined inductively on $k$ by joining every simplex of $(\Delta_{k-1} \times \{0\}) \cup (\partial \Delta_k \times [0,1])$ with the barycenter of $\Delta_k \times \{1\}$, as in the following figure. 

\begin{center}
\includegraphics[width=0.5\textwidth]{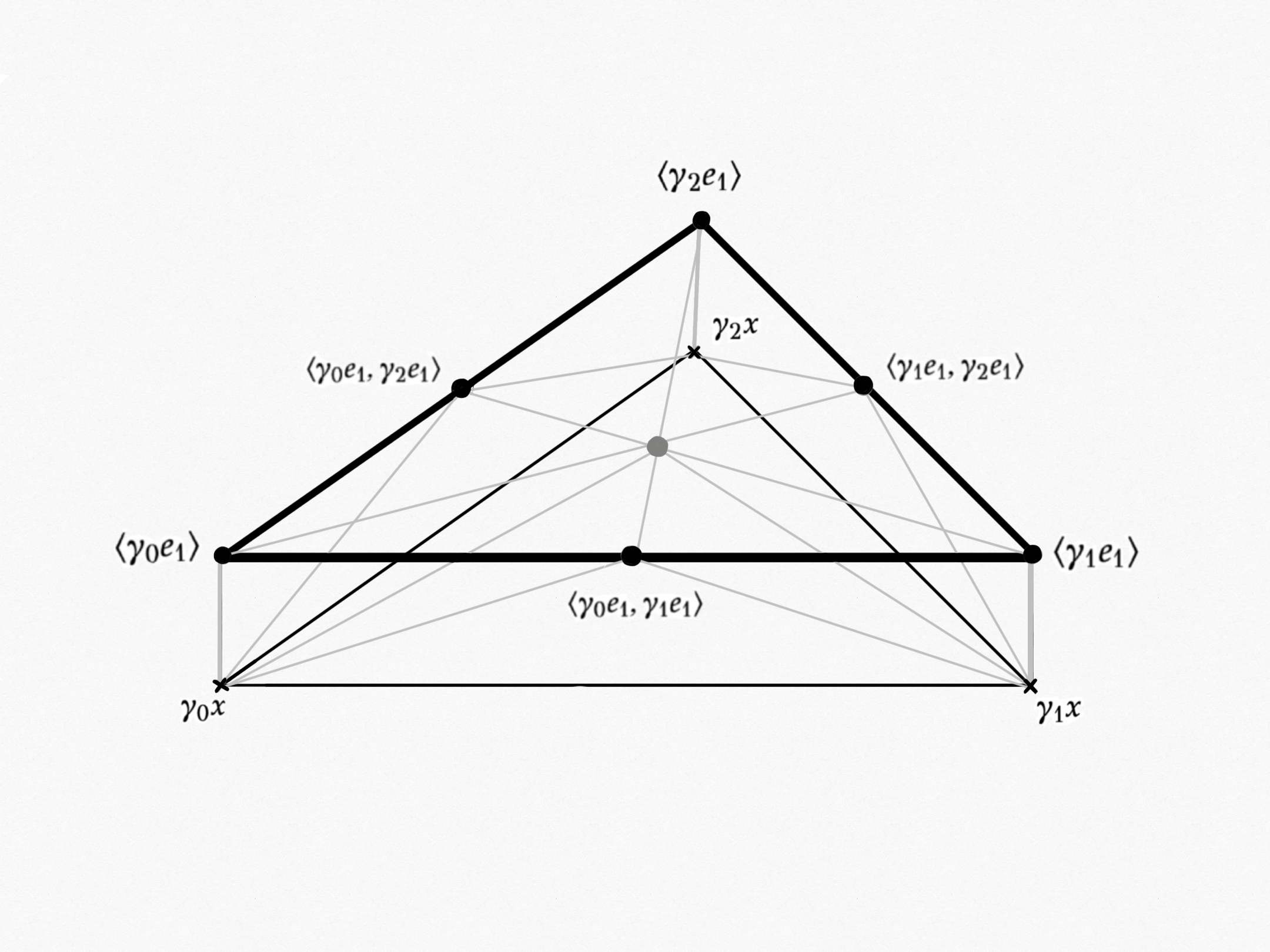}
\end{center}

More precisely, we define, inductively on $k$,  a subset $S_k$ of $\Delta_k \times \Delta_k '$ satisfying
\begin{equation} \label{eq:simplicial_decomposition_1}
|\Delta_k| \times |[0,1]| = \bigsqcup_{(s,s') \in S_k}^\circ C(|s| \times \{0\}, |s'| \times \{1\}),
\end{equation}
(here the symbol $\bigsqcup^\circ$ denotes an almost disjoint union: the cones indexed by different pairs in $S_k$ have disjoint interiors). We define $S_k$ as follows:
\begin{itemize}
\item[-] For $k=0$ we have $\Delta_0=\{*\}$ and we set $S_0=\Delta_0 \times \Delta_0$.

\item[-] Let $k>0$ and assume that $S_{k-1}$ has been defined. Let $x_0$ be the barycenter of $\Delta_k$. To describe which pairs $(s,s')$ belong to $S_k$, recall that every simplex $s' \in \Delta_k'$ either i) is the $0$-simplex $\{x_0\}$, or ii) lies on a face of $\partial \Delta_k$, or iii) is the cone $C(\{x_0\},s'')$ with vertex $x_0$ for a unique simplex $s''$ of $\Delta_k '$ contained in the boundary $\partial \Delta_k$. In case i) we declare that for every $s \in \Delta_k$ we have $(s,\{x_0\}) \in S_k$. In case ii), the vertex $s'$ lies on a face $\Delta_{k-1} \subset \partial \Delta_k$. We declare that $(s,s') \in S_k$ if and only if $s$ belongs to the same face of $\partial \Delta_k$ as $s'$ and $(s,s') \in S_{k-1}$. In case iii), we declare that $(s,s') \in S_k$ with $s'=C(\{x_0\},s'')$ if and only if $s$ and $s''$ belong to the same face of $\partial \Delta_k$ and $(s,s'') \in S_{k-1}$. Property \eqref{eq:simplicial_decomposition_1} follows by induction on $k$.
\end{itemize}

With this decomposition of $|\Delta_k| \times |[0,1]|$ in hand, we can now define $H$ by induction on $k$. For $k=0$ we recall the definition of $s(x,x'):[0,1] \to {_\mathbb{Q}}\overline{X}^T$ (see \eqref{eq:s_definition_1} and \eqref{eq:s_definition_2}). Writing $\langle \gamma_0 e_1 \rangle$ for the point of the boundary of $_{\mathbb{Q}}\overline{X}^T$ corresponding to the flag given by the line $\langle \gamma_0 e_1 \rangle$, we set
$$
H(\gamma_0,x) = s(\gamma_0 x, \langle \gamma_0 e_1\rangle).
$$
Note that the image of $H(\gamma_0,x)$ is the cone $C(\{\gamma_0 x\},\{\langle \gamma_0 e_1 \rangle \})$.

Next assume that $k \geq 1$ and that we have defined $H(\gamma_0,\ldots,\gamma_{k-1},x)$ for every collection of elements $\gamma_0,\ldots,\gamma_{k-1} \in \Gamma_0(\mathfrak{p},\Lambda(\mathfrak{I}))$. Let $\underline{\gamma}=(\gamma_0,\ldots,\gamma_k) \in \Gamma_0(\mathfrak{p},\Lambda(\mathfrak{I}))^{k+1}$. Assume first that $\langle \gamma_0 e_1,\ldots, \gamma_k e_1 \rangle \neq k^N$. Then $\Delta(\underline{\gamma})$ corresponds to a simplex in $\Delta_\mathbb{Q}(\mathbf{G})$. We define $H(\underline{\gamma},x)$ using the decomposition \eqref{eq:simplicial_decomposition_1} by taking the restriction of $H(\underline{\gamma},x)$ to $C(|s| \times \{0\}, |s'| \times \{1\})$ to be the simplicial map \eqref{eq:cones_as_simplicial_maps} whose image is the cone $C(\tilde{\Delta}(\underline{\gamma},x)(s),\Delta(\underline{\gamma})(s'))$.

Now assume that $\langle \gamma_0 e_1,\ldots, \gamma_k e_1 \rangle = k^N$ (and hence that $k+1=N$). Given $(s,s') \in S_{N-1}$, we define the restriction of $H(\underline{\gamma},x)$ to $C(|s| \times \{0\}, |s'|\times \{1\})$ to be the simplicial map \eqref{eq:cones_as_simplicial_maps} whose image is the cone defined as follows:
\begin{itemize}
\item[-] If $s'$ is the barycenter $x_0$ of $\Delta_{N-1}$, take the cone to be $C(\tilde{\Delta}(\underline{\gamma},x)(s),\{x_0(\underline{\gamma})\})$ (recall that $x_0(\underline{\gamma}) \in X$ denotes the barycenter of the modular symbol $\Delta(\underline{\gamma})$);

\item[-] If $s'$ belongs to the boundary $\partial \Delta_{N-1}$, then $s$ and $s'$ belong to the same face of $\partial \Delta_{N-1}$, and the restriction of $H(\underline{\gamma},x)$ to $C(|s| \times \{0\}, |s'|\times \{1\})$ has already been defined to be the map whose image is the cone $C(\tilde{\Delta}(\underline{\gamma},x)(s),\Delta(\underline{\gamma})(s'))$;

\item[-] In the remaining case we have $s'=C(\{x_0\},s'')$ for a unique simplex $s'' \in \partial \Delta_{N-1}$. In this case we form the cone $C':=C(\tilde{\Delta}(\underline{\gamma},x)(s),\{x_0(\underline{\gamma})\}) \subset X$ and take the cone to be $C(C',\Delta(\underline{\gamma})(s''))$.
\end{itemize}

By induction on $k$ one shows that $H(\underline{\gamma},x)$ is well-defined and continuous\footnote{This latter statement can be deduced from the following general principle: let $X$ and $Y$ two topological spaces, $(F_i )_{i \in I}$ a finite cover of $X$ by closed sets, and $f_i : F_i \to Y$ continuous maps. If $f_i$ and $f_j$ coincides on $F_i \cap F_j$ for all $i,j$, then there exists a (unique) continuous map 
$f : X \to Y$ that is equal to $f_i$ on $F_i$ for each $i$.} and satisfies \eqref{eq:H_homotopy_property} and \eqref{eq:homotopy_equivariance_1}. Note that the image of $H(\underline{\gamma},x)$ is given by a finite union of cones of the form $C(S,S')$, where $S$ is a compact subset of $X$ and $S'$ is a simplex in the boundary of ${_\mathbb{Q}}\overline{X}^T$ corresponding to a good cusp; it follows that the image of $H(\underline{\gamma},x)$ can be covered by finitely many Siegel sets attached to these cusps.

\subsection{Smoothing and evaluation}

We can  use the above results to express values of partial zeta functions as polynomials in Kronecker--Eisenstein series, by using the fact that the Eisenstein series $E_\mathfrak{p}(v_0;\psi,\Lambda(\mathfrak{I}))$ is closed and moving the simplices in \eqref{eq:fund_domain_decomposition} to the Tits boundary. In order to guarantee that the Eisenstein series is rapidly decreasing, we will use the following Lemma due to Colmez--Schneps \cite[Lemma 5]{ColmezSchneps}. In its statement we write $v_{\tilde{\mathfrak{P}}}$ for the valuation defined by a prime ideal $\tilde{\mathfrak{P}}$ of $\mathcal{O}_L$ and denote by $S_{L/k}$ the set of all non--zero prime ideals $\tilde{\mathfrak{P}}$ of $\mathcal{O}_L$ such that the residue field $\mathcal{O}_L/\tilde{\mathfrak{P}}$ has degree one over $\mathcal{O}/(\tilde{\mathfrak{P}} \cap \mathcal{O})$.

\begin{lemma} \label{lemma:Colmez_Schneps}
Let $\{\phi_i\}_{i \in I}$ be a finite collection of non-zero $k$-linear forms on $L$. There exists a constant $C$ such that if $\tilde{\mathfrak{P}} \in S_{L/k}$ satisfies  $\mathrm{N}\tilde{\mathfrak{P}} > C$ and $l \in L$ satisfies $v_{\tilde{\mathfrak{P}}}(l)<0$ and $v_{\mathfrak{P}'}(l) \geq 0$ for every other prime divisor $\mathfrak{P}'$ of $(\tilde{\mathfrak{P}} \cap \mathcal{O}) \mathcal{O}_L$, then $\phi_i(l) \neq 0$ for every $i \in I$.

In particular, if $\tilde{\mathfrak{P}} \in S_{L/k}$ satisfies $\mathrm{N}\tilde{\mathfrak{P}} > C$, $\mathfrak{a}$ is a fractional ideal of $L$ coprime to $(\tilde{\mathfrak{P}} \cap \mathcal{O}) \mathcal{O}_L$ and $l \in \mathfrak{a\tilde{P}}^{-1}-\mathfrak{a}$, then the forms $\phi_i$ are all non-vanishing on the coset $l+\mathfrak{a}$.
\end{lemma}
\begin{proof}
We can write $\phi_i(l)=\mathrm{tr}_{L/k}(l_i l)$ for unique $l_i \in L^\times$. Take $C$ so that $\mathrm{N}\tilde{\mathfrak{P}} >C$ implies that $\tilde{\mathfrak{p}}=\tilde{\mathfrak{P}} \cap \mathcal{O}$ is unramified in $L$ and for every prime divisor $\mathfrak{P}'$ of $\tilde{\mathfrak{p}}\mathcal{O}_L$ we have $v_{\mathfrak{P}'}(l_i)=0$ for all $i$. For $i \in I$ and $l$ as in the statement, we have $v_{\tilde{\mathfrak{P}}}(l_i l)<0$ and $v_{\mathfrak{P}'}(l_i l) \geq 0$ for every other prime divisor $\mathfrak{P}'$ of $\tilde{\mathfrak{p}}\mathcal{O}_L$. This implies (\cite[II \S  3, Cor. 2]{SerreCorpsLocaux}) that $\mathrm{tr}_{L/k}(l_i l)$ is not a $\tilde{\mathfrak{p}}$-integer, and hence is not zero.
\end{proof}

For a prime ideal $\tilde{\mathfrak{P}}$ of $\mathcal{O}_L$ coprime to $\mathfrak{f}$, $\mathfrak{a}$ and $\mathfrak{P}$, define the `$(\mathfrak{P},\tilde{\mathfrak{P}})$-smoothed' zeta function
$$
\zeta^{p,q}_{\mathfrak{f},\mathfrak{P},\tilde{\mathfrak{P}}}(\mathfrak{a},s)=\mathrm{N}\tilde{\mathfrak{P}}^{-s} \zeta^{p,q}_\mathfrak{f,\mathfrak{P}}(\mathfrak{a\tilde{P}},s)-\mathrm{N}\tilde{\mathfrak{P}}^{-s} \zeta^{p,q}_{\mathfrak{f},\mathfrak{P}}(\mathfrak{a},s).
$$
The following theorem implies \autoref{Theorem2Intro} of the introduction. 

\begin{theorem} \label{Tfinal}
There exists a constant $C$ such that if $\tilde{\mathfrak{P}}$ is a prime ideal of $\mathcal{O}_L$ such that the residue field $\mathcal{O}_L/\tilde{\mathfrak{P}}$ has degree one over $\mathcal{O}/\tilde{\mathfrak{p}}$ and $\mathrm{N}\tilde{\mathfrak{P}} >C$, then
$$
\det(\sigma_i(\alpha_j)) \zeta^{p,q}_{\mathfrak{f},\mathfrak{P},\tilde{\mathfrak{P}}}(\mathfrak{a},0) = [U(\mathfrak{f}):U(\mathfrak{f})']^{-1} \sum_{\sigma \in S_{N-1}} \mathrm{sgn}(\sigma) \sum_{\substack{l \in \tilde{\mathfrak{P}}^{-1} \mathfrak{f} /\mathfrak{f} \\ l \neq 0}} \mathbf{\Phi}^{pN,qN}_\mathfrak{p}(v_0+\alpha(l),\underline{u}_\sigma,\Lambda(\mathfrak{I}))(P_\alpha^{p,q}).
$$
\end{theorem}
\begin{proof}
First let us define a collection $\{\phi_i\}_{i \in I}$ as in \autoref{lemma:Colmez_Schneps}. Writing $\underline{u}_{\sigma,j}$ ($0 \leq j <N$) for the components of the $N$-tuple $\underline{u}_\sigma$ in \eqref{eq:def_u_sigma}, we consider the finite set $\{W_i\}_{i \in I}$ of all proper subspaces $W_i$ of $V_k$ of the form $\langle \underline{u}_{\sigma,j}e_1 \ | \ j \in J \rangle$, for all $\sigma \in S_{N-1}$ and all $J \subseteq \{0,\ldots,N-1\}$. For each subspace $W_i$ we choose a non-zero linear form $\phi_i$ on $L$ such that $W_i \subseteq \ker(\phi_i \circ \alpha^{-1})$. Let $C(I)$ be the constant provided by \autoref{lemma:Colmez_Schneps}. 

Now take $C>C(I)$ such that any prime ideal $\tilde{\mathfrak{P}}$ with $\mathrm{N}\tilde{\mathfrak{P}}>C$ is coprime to $\mathfrak{a}$, $\mathfrak{f}$ and $\mathfrak{P}$; then \autoref{lemma:Colmez_Schneps} and \autoref{thm:good_cusp_rapid_decrease} show that, for any $l \in  \mathfrak{\tilde{P}}^{-1} \mathfrak{f}-\mathfrak{f}$, the Eisenstein series $E_\mathfrak{p}(v_0+\alpha(l),P_{\alpha}^{pN,qN};\psi^{p,q} ,\Lambda(\mathfrak{I}))$ is rapidly decreasing on every Siegel set of every cusp corresponding to a flag $W_\bullet$ given by a chain of subpaces $W_i$ with $i \in I$; in particular, for such $l$ we have
$$
\int_{\partial H(\underline{u}_\sigma,x_\alpha)} E_\mathfrak{p}(v_0+\alpha(l),P_{\alpha}^{p,q};\psi^{pN,qN},\Lambda(\mathfrak{I})) = 0.
$$
Since 
$$
\sum_{\substack{l \in \mathfrak{\tilde{P}}^{-1} \mathfrak{f}/\mathfrak{f} \\ l \neq 0}} E_\mathfrak{p}(v_0+\alpha(l),P_\alpha^{p,q};\psi^{pN,qN},\Lambda(\mathfrak{I}))
$$
is invariant under $U(\mathfrak{f})^1$, the equivariance property \eqref{eq:homotopy_equivariance_1} shows that
\begin{equation} \label{eq:homotopy_Eis_computation}
\begin{split}
0 & = \sum_{\sigma \in S_{N-1}} \mathrm{sgn}(\sigma)  \int_{\partial H(\underline{u}_\sigma,x_\alpha)} \sum_{\substack{l \in \mathfrak{\tilde{P}}^{-1}\mathfrak{f}/\mathfrak{f} \\ l \neq 0}} E_\mathfrak{p}(v_0+\alpha(l),P_\alpha^{p,q};\psi^{pN,qN},\Lambda(\mathfrak{I})) \\
&= \sum_{\sigma \in S_{N-1}} \mathrm{sgn}(\sigma) \int_{\Delta^\circ(\underline{u}_\sigma)} \sum_{\substack{l \in \mathfrak{\tilde{P}}^{-1}\mathfrak{f}/\mathfrak{f} \\ l \neq 0}} E_\mathfrak{p}(v_0+\alpha(l),P_\alpha^{p,q};\psi^{pN,qN},\Lambda(\mathfrak{I})) \\
& \quad - \sum_{\sigma \in S_{N-1}} \mathrm{sgn}(\sigma) \int_{\tilde{\Delta}(\underline{u}_\sigma,x_\alpha)} \sum_{\substack{l \in \mathfrak{\tilde{P}}^{-1} \mathfrak{f}/\mathfrak{f} \\ l \neq 0}} E_\mathfrak{p}(v_0+\alpha(l),P_\alpha^{p,q};\psi^{pN,qN},\Lambda(\mathfrak{I})).
\end{split}
\end{equation}
The proof of \autoref{prop:Eisenstein_zeta_value} shows that
$$
[U(\mathfrak{f}):U(\mathfrak{f})']^{-1}\int_{X(\mathfrak{f})} \sum_{\substack{l \in \mathfrak{\tilde{P}}^{-1} \mathfrak{f}/\mathfrak{f} \\ l \neq 0}} E_\mathfrak{p}(v_0+\alpha(l),P_\alpha^{p,q};\psi^{pN,qN},\Lambda(\mathfrak{I})) = \det(\sigma_i(\alpha_j)) \zeta^{p,q}_{\mathfrak{f},\mathfrak{P},\tilde{\mathfrak{P}}}(\mathfrak{a},0).
$$
We compute
\begin{multline*}
\int_{X(\mathfrak{f})} \sum_{\substack{l \in \mathfrak{\tilde{P}}^{-1}\mathfrak{f}/\mathfrak{f} \\ l \neq 0}} E_\mathfrak{p}(v_0+\alpha(l),P_\alpha^{p,q};\psi^{pN,qN},\Lambda(\mathfrak{I})) \\
\begin{split}
&= \sum_{\sigma \in S_{N-1}} \mathrm{sgn}(\sigma) \int_{\tilde{\Delta}(\underline{u}_\sigma,x_\alpha)} \sum_{\substack{l \in \mathfrak{\tilde{P}}^{-1}\mathfrak{f}/\mathfrak{f} \\ l \neq 0}} E_\mathfrak{p}(v_0+\alpha(l),P_\alpha^{p,q};\psi^{pN,qN},\Lambda(\mathfrak{I}))  \quad (\text{by } \eqref{eq:fund_domain_decomposition}) \\
&= \sum_{\sigma \in S_{N-1}} \mathrm{sgn}(\sigma) \int_{\Delta^\circ(\underline{u}_\sigma)} \sum_{\substack{l \in \mathfrak{\tilde{P}}^{-1}\mathfrak{f}/\mathfrak{f} \\ l \neq 0}} E_\mathfrak{p}(v_0+\alpha(l),P_\alpha^{p,q};\psi^{pN,qN},\Lambda(\mathfrak{I}))  \quad (\text{by } \eqref{eq:homotopy_Eis_computation}) \\
&= \sum_{\sigma \in S_{N-1}} \mathrm{sgn}(\sigma) \sum_{\substack{l \in \mathfrak{\tilde{P}}^{-1} \mathfrak{f}/\mathfrak{f} \\ l \neq 0}} \mathbf{\Phi}^{pN,qN}_\mathfrak{p}(v_0+\alpha(l),\underline{u}_\sigma,\Lambda(\mathfrak{I}))(P_\alpha^{p,q}) \quad (\text{by Prop. } \ref{prop:Eis_integral_on_modular_symbols}).
\end{split}
\end{multline*}
\end{proof}

\medskip
\noindent
{\it Acknowledgments.} We would like to thank Don Blasius for patiently explaining us the fine algebraicity consequences that could be derived from our Theorem \ref{Tfinal}. We also thank Javier Fres\'an as well as our collaborator on a related project Akshay Venkatesh for their comments and corrections on this paper. The second author is partially supported by the Agence Nationale de la Recherche (ANR-18-CE40-0029 grant)  and by  from Sorbonne Universit\'e (Emergence DELCO grant). All three authors thank the anonymous referees for their careful reading of the manuscript.

\medskip

\bibliographystyle{plain}

\bibliography{bibli}

\end{document}